\documentclass{article}

\usepackage{microtype}
\usepackage{graphicx}
\usepackage{subfigure}
\usepackage{booktabs} 
\usepackage[table,xcdraw]{xcolor}

\usepackage{hyperref}
\hypersetup{colorlinks,citecolor=blue, urlcolor=cyan
}

\usepackage[accepted]{icml2024}

\usepackage{amsmath}
\usepackage{amssymb}
\usepackage{mathtools}
\usepackage{amsthm}
\usepackage{bbm}
\usepackage{xspace}
\usepackage{enumitem}

\usepackage[capitalize,noabbrev]{cleveref}

\theoremstyle{plain}
\newtheorem{theorem}{Theorem}[section]
\newtheorem{proposition}[theorem]{Proposition}
\newtheorem{lemma}[theorem]{Lemma}
\newtheorem{corollary}[theorem]{Corollary}
\theoremstyle{definition}

\newtheorem{assumption}[theorem]{Assumption}
\theoremstyle{remark}
\newtheorem{remark}[theorem]{Remark}

\usepackage[textsize=tiny]{todonotes}

\newtheorem{fact}{Fact}
\newtheorem{example}{Example}

\DeclareMathOperator*{\argmin}{argmin}

\DeclareMathOperator{\Graph}{Gra}
\DeclareMathOperator{\Zeros}{Zer}
\DeclareMathOperator{\Fixpoint}{Fix}
\DeclareMathOperator{\proximal}{prox}

\newcommand{\notshow}[1]{{}}
\newcommand{\AutoAdjust}[3]{{ \mathchoice{ \left #1 #2  \right #3}{#1 #2 #3}{#1 #2 #3}{#1 #2 #3} }}
\newcommand{\Xcomment}[1]{{}}

\newcommand{\InParentheses}[1]{\AutoAdjust{(}{#1}{)}}
\newcommand{\InBrackets}[1]{\AutoAdjust{[}{#1}{]}}
\newcommand{\InAngles}[1]{\AutoAdjust{\langle}{#1}{\rangle}}
\newcommand{\InNorms}[1]{\AutoAdjust{\|}{#1}{\|}}

\newcommand{\R}{\mathbbm{R}}

\newcommand{\X}{\mathcal{X}}
\newcommand{\Y}{\mathcal{Y}}
\newcommand{\Z}{\mathcal{Z}}

\newcommand{\half}{\frac{1}{2}}

\newcommand{\boldzero}{\boldsymbol{0}}

\newcommand{\indSet}{\mathbbm{I}}
\newcommand{\real}{\mathbbm{R}}

\newcommand{\gap}{\textsc{Gap}}

\newcommand{\LHSI}{\text{LHS of Inequality}}

\newcommand{\ball}[2]{\mathcal{B}({#1},{#2})}

\newenvironment{nalign}{    
\begin{equation}
    \begin{aligned}
}{
    \end{aligned}
    \end{equation}
    \ignorespacesafterend
}

\def\+#1{\mathcal{#1}}
\def\-#1{\mathbb{#1}}
\newcommand{\Ham}{r^{tan}}

\newcommand{\CMI}{\ref{CMI}\xspace}

\usepackage{pifont}
\newcommand{\cmark}{\ding{51}}%
\newcommand{\xmark}{\ding{55}}%
\icmltitlerunning{Accelerated Algorithms for Constrained Nonconvex-Nonconcave Min-Max Optimization and Comonotone Inclusion}

\begin{document}

\twocolumn[
\icmltitle{Accelerated Algorithms for Constrained Nonconvex-Nonconcave Min-Max Optimization and Comonotone Inclusion}



\icmlsetsymbol{equal}{*}

\begin{icmlauthorlist}
\icmlauthor{Yang Cai}{equal,yyy}
\icmlauthor{Argyris Oikonomou}{equal,yyy}
\icmlauthor{Weiqiang Zheng}{equal,yyy}
\end{icmlauthorlist}

\icmlaffiliation{yyy}{Department of Computer Science, Yale University, New Haven, USA}

\icmlcorrespondingauthor{Yang Cai}{yang,cai@yale.edu}
\icmlcorrespondingauthor{Argyris Oikonomou}{argyris.oikonomou@yale.edu}
\icmlcorrespondingauthor{Weiqiang Zheng}{weiqiang.zheng@yale.edu}

\icmlkeywords{Machine Learning, ICML, Minimax Optimization, Last-Iterate Convergence}

\vskip 0.3in
]



\printAffiliationsAndNotice{\icmlEqualContribution} 

\begin{abstract}
We study constrained comonotone min-max optimization, a structured class of nonconvex-nonconcave min-max optimization problems, and their generalization to comonotone inclusion. 
In our first contribution, we extend the \emph{Extra Anchored Gradient} (EAG) algorithm, originally proposed by~\citet{yoon2021accelerated} for unconstrained min-max optimization, to constrained comonotone min-max optimization and comonotone inclusion, achieving an optimal convergence rate of $O\left(\frac{1}{T}\right)$ among all first-order methods. Additionally, we prove that the algorithm's iterations converge to a point in the solution set. In our second contribution, we extend the \emph{Fast Extra Gradient} (FEG) algorithm, as developed by \citet{lee2021fast}, to constrained comonotone min-max optimization and comonotone inclusion, achieving the same $O\left(\frac{1}{T}\right)$ convergence rate. This rate is applicable to the broadest set of comonotone inclusion problems yet studied in the literature. Our analyses are based on simple potential function arguments, which might be useful for analyzing other accelerated algorithms. 
\end{abstract}

\section{Introduction}\label{intro}
Min-max optimization is a cornerstone in game theory, optimization, and online learning. While classical theoretical studies have primarily focused on the convex-concave case,\footnote{For closed convex sets $\X$ and $\Y$, and a smooth function $f(\cdot,\cdot)$, the corresponding min-max optimization problem is formulated as:
$
    \min_{x\in \X} \max_{y\in \Y} f(x,y),
$. A min-max optimization problem is convex-concave when additionally, function $f(\cdot,\cdot)$ is convex in $x$ and concave $y$.}   there has been growing interest in nonconvex-nonconcave min-max optimization problems within the machine learning and optimization community. This surge in attention is attributed to a range of innovative applications of
nonconvex-nonconcave min-max such as generative adversarial networks (GANs) (see~\citep{goodfellow_generative_2014,arjovsky_wasserstein_2017}), adversarial examples (see~\citep{madry_towards_2017}), robust optimization (see~\citep{ben-tal_robust_2009}), and reinforcement learning (see~\citep{du_stochastic_2017,dai_sbeed_2018}). 

Unfortunately, even finding a first-order stationary point is generally intractable for nonconvex-nonconcave min-max optimization problems~ \citep{daskalakis2021complexity}. In light of this obstacle, recent research has shifted focus towards classes of nonconvex-nonconcave min-max optimization problems that exhibit more structure. One class that has  received extensive attention is the comonotone min-max optimization problem proposed by \citet{bauschke2021generalized} (see \Cref{ex:minmax} for the definition), which captures the convex-concave setting as a special case. 

During the past few years, we have witnessed a series of exciting advancements in comonotone min-max optimization~\citep{diakonikolas2021efficient,pethick2022escaping,lee2021fast,yoon2022accelerated,InIterate23,Z23-ogda,Z23-single,tran2023extragradient,GorbunovTHG23}. However, the algorithms proposed in these studies often exhibit one or more of the following  limitations: (a) they suffer a \emph{sub-optimal} convergence rate, (b) they only apply to the \emph{unconstrained setting} or (c) they \emph{lack point convergence guarantees}.\footnote{An algorithm has point convergence if its iterates has a limit.} Addressing these challenges, the primary goal of our paper is to design algorithms that simultaneously surmount all three of these limitations. More specifically, we seek to answer the following question:

\smallskip\noindent
\null\hfill
\begin{minipage}{0.45\textwidth}
\textbf{Question:} Can we design algorithms for comonotone min-max optimization that have (i) the \textbf{optimal convergence rate} in the \textbf{constrained setting} and (ii) \textbf{point convergence} to the solution set?
\end{minipage}
\hfill\null
\smallskip

We provide an affirmative answer to this question. In particular, we design algorithms that exhibit the optimal convergence rate of $O(1/T)$ in the constrained setting for comonotone min-max optimization, and our algorithms are guaranteed to converge to a point within the solution set. See \Cref{tab:table} for comparison between our algorithms and the ones proposed in the literature. Indeed, our findings extend to the more general problem of \emph{Comonotone  Inclusion}. More specifically, we consider composite inclusion problems that involve a single-valued and Lipschitz continuous operator $F: \mathbb{R}^n \rightarrow \mathbb{R}^n$, and a maximally monotone set-valued operator $A: \mathbb{R}^n \rightarrow \mathcal{P}(\mathbb{R}^n)$.\footnote{The powerset of $\mathbb{R}^n$ is denoted as $\mathcal{P}(\mathbb{R}^n)$.} The goal is to find a point $z^*\in \real^n$ such that $\boldzero \in F(z^*)+ A(z^*)$. If the composite operator satisfies the $\rho$-comonotonicity condition:
\begin{align*}
    &\InAngles{u - u',z-z'} \ge \rho \InNorms{u-u'}^2, \\
     &\quad\forall z, z' \in \mathbb{R}^n \text{ and } \forall u \in F(z)+A(z), u' \in F(z') + A(z'),
\end{align*}
we refer to these problems as Comonotone Inclusion Problems (\CMI). It is not hard to see that both the unconstrained comonotone min-max optimization and the monotone variational inequality fall under the umbrella of \CMI. The general formulation of \CMI enables us to capture more complex settings, such as constrained or non-smooth min-max optimization. For a more detailed discussion, see \Cref{sec:prelim}.

\notshow{\footnote{Given a $\rho$-comonotone min-max optimization problem on convex sets $\X$ and $\Y$, if we set $\Z = \X\times \Y$,  and $F(x,y) = \begin{pmatrix}
  \nabla_x f(x,y)\\
  -\nabla_y f(x,y)
\end{pmatrix}$, then (i) $F(x,y)$ is a Lipschitz and $\rho$-comonotone operator, and (ii) the set of saddle points coincide with the solutions to Comonotone Inclusion Problem~\CMI problem for operators $F$ and $A = \partial \indSet_\Z$ (where $\partial \indSet_\Z$ is the subdifferential operator of  the the indicator function for set $\Z$, $\indSet_\Z(\cdot)$). }}

\subsection{Our Contributions}\label{sec:contribution}
In this paper, we evaluate the quality of a solution by measuring the norm of the composite operator. Specifically, a point $z\in\real^n$ is an $\epsilon$-approximate solution if $ \boldzero\in F(z)+A(z)+\ball{\boldzero}{\epsilon},$ where $\ball{\boldzero}{\epsilon}$ is the ball with radius $\epsilon$ centered at $\boldzero$. As discussed in  Section~\ref{sec:convergence criteria}, this criterion is equivalent to the tangent residual of $z$, a notion introduced in~\cite{cai2022finite}, being no more than $\epsilon$.

First, we extend the Extra Anchored Gradient algorithm \eqref{composite-EAG}, originally proposed by~\citet{yoon2021accelerated} for unconstrained convex-concave min-max problems, to solve $\rho$-comonotone composite inclusion problems. 

\smallskip\noindent
\hspace{.15in}\begin{minipage}{0.45\textwidth}\textbf{Contribution 1:}  We show that the $T$-th iterate of \eqref{composite-EAG} is an $O(\frac{1}{T})$-approximate solution to \CMI if (i) the single-valued operator $F$ is $L$-Lipshitz, (ii) the set-valued operator $A$ is maximally monotone, and (iii) the composite operator $F+A$ is $\rho>-\frac{1}{20L}$-comonotone. Additionally, our algorithm ensures point convergence to a point within the solution set when condition (i), (ii) and (iii) hold.
\end{minipage}
\smallskip
\noindent 

Next, we extend the Fast Extra-Gradient (FEG) algorithm~\citet{lee2021fast} originally designed for unconstrained comonotone convex-concave min-max optimization to comonotone inclusion problems. 
\smallskip\noindent
\hspace{.15in}\begin{minipage}{0.45\textwidth}\textbf{Contribution 2:} We show that the $T$-th iterate of \eqref{composite-FEG} is an $O(\frac{1}{T})$-approximate solution to \CMI if (i) the single-valued operator $F$ is $L$-Lipschitz, (ii) the set-valued operator $A$ is maximally monotone, and (iii) the composite operator $F+A$ is $\rho>-\frac{1}{2L}$-comonotone. Additionally, our algorithm is guaranteed point convergence if (i) and (ii) hold, and the composite operator $F+A$ is monotone, i.e., $0$-comonotone.
\end{minipage}

Note that the convergence rate of our algorithms \ref{composite-EAG} and \ref{composite-FEG} matches the lower bound by \cite{diakonikolas2020halpern,yoon2021accelerated}, and is therefore optimal for any first-order method. Some further remarks are in order.
\begin{itemize}[leftmargin=*]
    \item Our algorithms \ref{composite-EAG} and \ref{composite-FEG} are the \emph{first} first-order algorithms with optimal \(O\left(\frac{1}{T}\right)\) convergence rates for the constrained and negatively comonotone setting. We remark that \citet{kovalev2022first} concurrently propose \ref{composite-EAG} and prove its convergence in the constrained and monotone setting.
    \item Our algorithm~\ref{composite-FEG} achieves the optimal convergence rate in the constrained and negatively monotone setting with \(\rho > -\frac{1}{2L}\), which represents the widest range of \(\rho\) values among all single-loop algorithms.\footnote{Very recently, some works propose algorithms that work for a wider range of $\rho > -\frac{1}{L}$ in the unconstrained~\citep{cai2024variance, fan2024weaker} and constrained setting~\cite{alacaoglu2024extending}. However, these algorithms are \emph{double-loop} and are sub-optimal by $\log T$ in terms of the convergence rates. It is an open question to design a single-loop algorithm that works for $\rho > -\frac{1}{L}$ with no spurious log factors.}
    \item We are the \emph{first} to provide first-order methods that achieve both optimal convergence rates and point convergence in the constrained and negatively comonotone setting.
\end{itemize}

\begin{table*}[t]
\centering
\caption{Existing results for comonotone inclusion problem. The convergence rate is in terms of the operator norm (in the unconstrained setting) and the residual (in the constrained setting).  ($\dagger$): The $O(\frac{1}{\sqrt{T}})$ convergence rates hold under the weak Minty assumption, which is implied by the comonotonicity assumption.}
\label{tab:table}

\begin{tabular}{@{}ccccc@{}}
\toprule
Algorithm &
  Rate &
  Range of $\rho$ &
  Constraints &
  Point Convergence \\ \midrule
\begin{tabular}[c]{@{}c@{}}EG+~\citep{diakonikolas2021efficient}\\ \citep{GorbunovTHG23}$\dagger$\end{tabular} &
  \cellcolor[HTML]{C0C0C0}$O\InParentheses{\frac{1}{\sqrt{T}}}$ &
  $\left(- \frac{1}{2L},+\infty\right)$ &
  \cellcolor[HTML]{C0C0C0}\xmark &
  \cmark \\ \midrule
CEG+~\citep{pethick2022escaping,stablencnc}$\dagger$ &
  \cellcolor[HTML]{C0C0C0}$O\InParentheses{\frac{1}{\sqrt{T}}}$ &
  $\left(- \frac{1}{2L},+\infty\right)$ &
  \cmark &
  \cmark \\ \midrule
\begin{tabular}[c]{@{}c@{}}Double-Loop Halpern~\citep{diakonikolas2020halpern} \\ \citep{kohlenbach2022proximal}\end{tabular} &
  \cellcolor[HTML]{C0C0C0}$O(\frac{\log\InParentheses{T}}{T})$ &
  $ (-\frac{1}{2L},+\infty]$ &
  \cmark &
  \cellcolor[HTML]{C0C0C0}Unknown \\ \midrule
\text{APG}$^*$~\citep{yoon2022accelerated} &
  \cellcolor[HTML]{C0C0C0}$O(\frac{\log\InParentheses{T}}{T})$ &
  \cellcolor[HTML]{C0C0C0}$ [0,+\infty]$ &
  \cmark &
  \cmark \\ \midrule 
EAG~\citep{yoon2021accelerated} &
  $O(\frac{1}{T})$ &
  \cellcolor[HTML]{C0C0C0}$ [0,+\infty]$ &
  \cellcolor[HTML]{C0C0C0}\xmark &
  \cmark \\ \midrule
FEG~\citep{lee2021fast} &
  $O(\frac{1}{T})$ &
  $\left(- \frac{1}{2L},+\infty\right)$ &
  \cellcolor[HTML]{C0C0C0}\xmark &
  \cmark (for $\rho\geq 0)$ \\ \midrule
ARG~\citep{Z23-single} &
  $O(\frac{1}{T})$ &
  $\left[- \frac{1}{60L},+\infty\right)$ &
  \cmark &
  \cellcolor[HTML]{C0C0C0}Unknown \\ \midrule
AOG~\citep{Z23-ogda} &
  $O(\frac{1}{T})$ &
  \cellcolor[HTML]{C0C0C0}$ [0,+\infty]$ &
  \cmark &
  \cellcolor[HTML]{C0C0C0}Unknown \\ \midrule
fOGDA~\citep{InIterate23} &
  $O(\frac{1}{T})$ &
  \cellcolor[HTML]{C0C0C0}$ [0,+\infty]$ &
  \cmark &
  \cmark \\ \midrule \midrule
\ref{composite-EAG}~[{\bf This paper}] &
  $O(\frac{1}{T})$ &
  $\left[- \frac{1}{20L},+\infty\right)$ &
  \cmark &
  \cmark \\ \midrule
\ref{composite-FEG}~[{\bf This paper}] &
  $O(\frac{1}{T})$ &
  $\left(- \frac{1}{2L},+\infty\right)$ &
  \cmark &
  \cmark (for $\rho\geq 0)$ \\ \bottomrule
\end{tabular}%

\end{table*}

\subsection{Related Work}\label{sec:related work}
There is a vast literature on {inclusion problems} and variational inequalities. See~\citep{facchinei_finite-dimensional_2003,bauschke_convex_2011,ryu2016primer} and the references therein. We provide a brief discussion of the most relevant results here and postpone the rest to \Cref{apx:related work}.

\paragraph{Accelerated Convergence Rate for Convex-Concave Min-Max Optimization and Monotone Variational Inequalities.} We overview results that achieve the $O(\frac{1}{T})$ accelerated convergence rate in terms of the operator norm or residual. Note that these results imply $O(\frac{1}{T})$ convergence rate in terms of the gap function. We describe these results in the language of inclusion problems. 

Recent results show accelerated rates through Halpern iteration~\citep{halpern1967fixed} or a similar mechanism -- anchoring. Implicit versions of Halpern iteration have $O(\frac{1}{T})$ convergence rate~\citep{kim_accelerated_2021,lieder2021convergence,park2022exact} for monotone operators and explicit variants of Halpern iteration achieve the same convergence rate when $F$ is cocoercive and $A$ is the subdifferential of the indicator function of the feasible set~\citep{diakonikolas2020halpern,kim_accelerated_2021}.
\citet{diakonikolas2020halpern} also provides a double-loop implementation of the algorithm for monotone operators at the expense of an additional logarithmic factor in the convergence rate.    
\citet{yoon2021accelerated} proposes the extra anchored gradient (EAG) method, which is the first explicit method with an accelerated $O(\frac{1}{T})$ rate in the unconstrained setting for monotone operators, i.e., $F$ is monotone and $A=0$.\footnote{We use $A=0$ to denote the set operator that maps any point to $\emptyset$}. They also establish a matching $\Omega(\frac{1}{T})$ lower bound that holds for all first-order methods.
Convergence analysis of the past extragradient method with anchoring in the unconstrained setting is provided by \citet{tran2021halpern}. More recently, \citet{tran2022connection} studies the connection between Halpern iteration and Nesterov's accelerated method and provides new algorithms for monotone 
operators and alternative analyses for EAG in the unconstrained setting.  \citet{InIterate23} analyzes the fast optimistic gradient descent ascent (fOGDA) and shows that the algorithms achieve $O(\frac{1}{T})$ convergence rate for constrained monotone variational inequalities, i.e., $F$ is monotone and $A$ is the subdifferential of the indicator function of the feasible set.  \citet{Z23-ogda} proposes an anchored variant of optimistic gradient descent ascent (AOG) for monotone games, and their analysis extends to constrained monotone variational inequalities.

\paragraph{Accelerated Convergence Rate for Comonotone Min-Max Optimization and Comonotone Inclusion.} 
For unconstrained comonotone min-max optimization problems, i.e., $F$ is comonotone and $A=0$, \citet{lee2021fast} proposes a generalization of EAG called fast extraradient (FEG), which achieves $O(\frac{1}{T})$ convergence rate. For constrained comonotone min-max optimization problems, convergence of \emph{implicit methods} such as the proximal point algorithm and the Halpern iteration has been analyzed~\citep{bauschke2021generalized,kohlenbach2022proximal}. Moreover, the explicit double-loop variants of the Halpern iteration achieves $O(\frac{\log T}{T})$ convergence rates~\citep{diakonikolas2020halpern, kohlenbach2022proximal}. \citet{Z23-single} analyzes the anchored variant of Reflected Gradient (ARG) and shows a $O(\frac{1}{T})$ fast convergence rate for a subset of comonotone inclusion problems covered by our algorithms.\footnote{See the paragraph of ``Preliminary Version of this Paper'' for discussion between their results and ours.} It is unknown whether their algorithm has point convergence.

\paragraph{Point Convergence in the Convex-Concave and Monotone Setting.} For unconstrained convex-concave min-max optimization problems, \citet{yoon2022accelerated} shows that a range of fast algorithms, namely EAG \cite{yoon2021accelerated}, FEG \citep{lee2021fast} and anchored Popov's scheme \citep{tran2021halpern}, all converge to the Optimized Halpern Iteration \citep{lieder2021convergence,kim_accelerated_2021} -- an implicit method. Consequentially, this implies that all aforementioned methods have point convergence, as the Optimized Halpern Iteration has point convergence. Additionally, \citet{InIterate23} shows that the fOGDA has point convergence to the solution set for constrained monotone variational inequalities.

\paragraph{Preliminary Version of this Paper.} We would like to mention that the initial version of this paper was first made available online in June 2022. This early version included an algorithm similar to \ref{composite-EAG} that is capable of achieving an $O(\frac{1}{T})$ accelerated convergence rate for \CMI. In concurrent work, \citet{kovalev2022first} propose \ref{composite-EAG} and prove its convergence for \CMI with $\rho \ge 0$. The current version not only includes this result but also introduces a novel point convergence analysis. Note that \citet{InIterate23,Z23-ogda,Z23-single} appeared after our manuscript was available online. In particular, the analyses used in \citet{Z23-ogda,Z23-single} are similar to the one first proposed in this paper, building upon certain key lemmas and inequalities established here. \citet{tran2023extragradient} provides an alternative proof for the convergence rates of our algorithm \ref{composite-FEG} but does not provide any point convergence guarantees.\footnote{\ref{composite-FEG} was proposed in the second version of this paper, which appeared online in August 2022.} 

We summarize all results discussed in this section and our results in Table~\ref{tab:table}. 

\section{Preliminaries}\label{sec:prelim}
We consider the Euclidean Space $(\R^n, \InNorms{\cdot})$, where $\InNorms{\cdot}$ is the $\ell_2$ norm and $\InAngles{\cdot,\cdot}$ denotes inner product on $\R^n$. 

\paragraph{Basic Notions about Operators.} A set-valued operator $A:\R^n \rightrightarrows \R^n$ maps $z\in \R^n$ to a subset $A(z) \subseteq \R^n$. We say $A$ is \emph{single-valued} if $|A(z)| = 1$ for all $z \in \R^n$. The \emph{graph} of an operator $A$ is defined as $\Graph_A = \{(z, u): z\in \R^n, u \in A(z)\}$. The \emph{inverse operator} of $A$ is denoted as $A^{-1}$ whose graph is $\Graph_{A^{-1}} = \{(u,z): (z, u) \in \Graph_A\}$. We also denote $\Zeros(A) = \{z \in \R^n: 0 \in A(z)\}$. For two operators $A$ and $B$, we denote $A + B$ as the operator with graph $ \Graph_{A+B} = \{(z, u_A + u_B): (z, u_A) \in \Graph_A, (z, u_B) \in \Graph_B\}$.
We denote the identity operator as $I : \R^n \rightarrow \R^n$. For $L \in (0, \infty)$, a single-valued operator $A : \R^n \rightarrow \R^n $ is \emph{$L$-Lipschitz} if 
\begin{align*}
    \InNorms{A(z) - A(z')} \le L \cdot \InNorms{z - z'}, \quad \forall z,z'\in \R^n. 
\end{align*}
Moreover, $A$ is \emph{non-expansive} if it is $1$-Lipschitz. For a non-expansive operator $A$, we denote its fix points $\Fixpoint(A) = \{z \in \R^n: z = A(z)\}$.  For a closed convex set $\Z \subseteq \R^n$ and point $z \in \R^n$, we denote the \emph{normal cone operator} as $N_\Z$:
\begin{align*}
    N_\Z(z) = 
    \begin{cases}
    \emptyset, & z \notin \Z, \\
    \{v \in \R^n : \InAngles{v, z' - z} \le 0, \forall z' \in \Z\}, & z \in \Z.
    \end{cases}
\end{align*}
Define the indicator function \begin{equation*}
\indSet_\Z(z) = \begin{cases}
0 & \quad \text{if }z\in\Z, \\
+\infty & \quad \text{otherwise}.
\end{cases}
\end{equation*}
Then it is not hard to see that the subdifferential operator $\partial \indSet_\Z = N_\Z$. The projection operator $\Pi_\Z:\R^n \rightarrow \R^n$ is defined as $\Pi_\Z [z] := \argmin_{z' \in \Z} \InNorms{z - z'}^2$.

\paragraph{(Co)monotone Operators} For $\rho \in \R$, an operator $A: \R^n \rightrightarrows \R^n$ is $\rho$-comonotone~\citep{bauschke2021generalized},\footnote{ when $\rho < 0$, $\rho$-comonotonicity is also known as $|\rho|$-cohypomonotonicity~\citep{combettes2004proximal}.} if 
\begin{align*}
    \InAngles{u - u', z - z'} \ge \rho \InNorms{u-u'}^2, \quad \forall (z,u),(z',u') \in \Graph_A.
\end{align*}
If $A$ is $0$-comonotone, then $A$ is \emph{monotone}. If $A$ is $\rho$-comonotone for $\rho > 0$, we also say $A$ is \emph{$\rho$-cocoercive} (a stronger assumption than monotonicity). When $A$ satisfies negative comonotonicity, i.e., $\rho$-comonotonicity with $\rho < 0$, then $A$ is possibly non-monotone.  Negative comonotonicity is the focus of this paper in the non-monotone setting. 

$A$ is \emph{maximally $\rho$-comonotone} if $A$ is $\rho$-comonotone and there is no other $\rho$-comonotone operator $B$ such that $\Graph_A \subset \Graph_B$ strictly. $A$ is maximally monotone if it is maximally $0$-comonotone. When $f:\R^n \rightarrow \R$ is a convex closed proper function, then its subdifferential operator $\partial f$ is maximally monotone. As an example, $\partial \indSet_\Z = N_\Z$ is maximally monotone. 
\paragraph{Resolvent Operators} We denote the \textbf{resolvent} of an operator $A$ as $J_{A}:= (I + A)^{-1}$. 
\begin{proposition}\citep{rockafellar_monotone_1976,ryu2016primer,bauschke_convex_2011, Ryu22Large}
    When $A$ is maximally monotone, its resolvent $J_{\eta A}$ has the following properties:
     for any $\eta > 0$, \begin{enumerate}
    \item $J_{\eta A}$ is well-defined on $\R^n$ and  non-expansive;
    \item if $z = J_{\eta A}(z')$,  $\frac{z' - z}{\eta} \in A(z)$;
    \item when $A = N_\Z$ is the normal cone operator of a closed convex set $\Z \subseteq \R^n$, $J_{\eta A} = \Pi_\Z$ is the projection operator.
\end{enumerate}
\end{proposition}

\subsection{Inclusion Problems with Negatively Comonotone Operators.} 
Given a single-valued and possibly \emph{non-monotone} operator $F$ and a \emph{set-valued maximally monotone} operator $A$, we denote $E = F + A$. The comonotone inclusion problem is to find a point $z^* \in \R^n$ that satisfies 
\begin{equation}\label{CMI}\tag{CMI}
    \begin{gathered}
        \boldsymbol{0} \in E(z^*) = F(z^*) + A(z^*).
    \end{gathered}
\end{equation}
We say $z$ is an $\epsilon$-approximate solution to \CMI if $\boldsymbol{0} \in F(z) + A(z) + \mathcal{B}(\boldsymbol{0}, \epsilon)$. We summarize the assumptions on \CMI below.
\begin{assumption}\label{assumption:CMI}
In \CMI, 
    \begin{enumerate}
        \item $F:\R^n \rightarrow \R^n$ is $L$-Lipschitz. 
        \item $A:\R^n \rightrightarrows \R^n$ is maximally monotone.
        \item $E = F+A$ is $\rho$-comonotone, i.e., there exists $\rho \le 0$ such that 
        \begin{align*}
            \InAngles{u - u', z - z'} \ge \rho \InNorms{u - u'}^2, \forall (z, u), (z', u') \in \Graph_E.
        \end{align*}
        \item There exists a solution $z^* \in \R^n$ such that $\boldsymbol{0} \in E(z^*)$.
    \end{enumerate}
\end{assumption}
The formulation of \CMI provides a unified treatment for a range of problems including variational inequalities (see \Cref{sec:preliminary-MI&VI}), min-max optimization, and multi-player games. We present one detailed example below and refer readers to \citep{facchinei_finite-dimensional_2003} for more examples.
\begin{example}[Min-Max Optimization]\label{ex:minmax}
The following structured min-max optimization problem captures a wide range of applications in machine learning such as GANs, adversarial examples, robust optimization, and reinforcement learning:
\begin{align}\label{eq:min-max}
    \min_{x \in \R^{n_x}} \max_{y \in \R^{n_y}} f(x,y) + g(x) - h(y),
\end{align}
where $f(\cdot,\cdot)$ is possibly non-convex in $x$ and non-concave in $y$. Regularized and constrained min-max problems are covered by appropriate choices of lower semicontinuous and convex functions $g$ and $h$. Examples include $\ell_1$-norm, $\ell_2$-norm, and indicator function of a convex feasible set. Let $z=(x,y)$, we define $F(z) = (\partial_x f(x,y), -\partial_y f(x,y))$ and $A(z) = (\partial g(x), \partial h(y))$. If $F$ and $A$ satisfies \Cref{assumption:CMI}, then we call it a \emph{comonotone min-max optimization} problem. See \citep[Example 1]{lee2021fast} for examples of nonconvex-nonconcave conditions that imply negative comonotonicity. We also note that the resolvent of $\partial g$ is also known as the \emph{proximal operator} $\proximal_{\eta g} = (I + \lambda \eta \partial g)^{-1}$. The resolvent $J_{\eta \partial g}$ is efficiently computatable for $g$ being the $\ell_1$-norm, $\ell_2$-norm, matrix norms, an indicator function for a convex feasible set, among many others (see \citep[Ch 6, 7]{parikh2014proximal} for a comprehensive overview of proximal operators and their efficient computation).
\end{example}

\subsection{Convergence Criteria}\label{sec:convergence criteria}
An appropriate convergence criterion is the \emph{tangent residual} $\Ham_{F,A}(z) := \min_{c \in A(z)} \InNorms{F(z) + c}$, which is an extension of the operator norm $\InNorms{F(z)}$ in the unconstrained setting ($A = 0$) to the composite setting. This is a natural quantity for inclusion problems, as $\Ham_{F,A}(z) \le \epsilon$ is equivalent to $z$ being an $\epsilon$-approximate solution to \CMI. In this paper, following \citep{cai2022finite}, we refer to this quantity as the ``tangent residual". Another commonly-used convergence criterion that captures the stationarity of a solution is the \emph{natural residual} $r^{nat}_{F,A}:= \InNorms{z - J_A\InBrackets{z - F(z)}}$. Note that $z^*$ is a solution to \CMI iff $z^* = J_A \InBrackets{z^*- F(z^*)}$. The following fact shows that natural residual is upper bounded by the tangent residual (see proof in \Cref{app:ts<=ns}).

\begin{fact}\label{fact:tangent residual smaller than natural residual}
In \CMI, $r^{nat}_{F,A}(z) \le \Ham_{F,A}(z)$. 
\end{fact} 

In this paper, we state our convergence rates in terms of the tangent residual $\Ham_{F,A}(z)$, which implies (i) convergence rates in terms of the natural residual $r_{F,A}^{nat}(z)$, and (ii) $z$ is an approximate solution to \CMI and the variational inequality problems including \ref{SVI} and \ref{MVI} (see Definitions in \Cref{sec:preliminary-MI&VI}).

\section{\ref{composite-EAG} for Comonotone Inclusion with Point Convergence}
\label{sec:composite-EAG}

\begin{algorithm*}[!ht]
    \caption{\ref{composite-EAG} for \CMI with $\rho > -\frac{1}{2L}$}
    \begin{algorithmic} \label{alg:composite-EAG}
        \STATE \textbf{Initialization:}  $z^0 \in \R^n$. Step size $\eta > 0$. $\beta_k = \frac{1}{k+1}$ for $k \ge 0$. Set \begin{align*}
            z_1 = J_{\eta A}[z_0 - \eta F(z_0)], \quad c_1 = \frac{z_0 - \eta F(z_0) - z_1}{\eta}
        \end{align*}
        \FOR{$k \ge 1$}
        \STATE \begin{equation}\label{composite-EAG}\tag{composite-EAG}
                \begin{aligned}
                    z_{k+\half} &= \beta_k z_0 + (1- \beta_k) z_k - \eta F(z_k) - \eta c_k \\
                    z_{k+1} &= J_{\eta A}[\beta_k z_0 + (1- \beta_k) z_k - \eta F(z_{k+\half})] \\
                    c_{k+1} &= \frac{\beta_k z_0 + (1- \beta_k) z_k - \eta F(z_{k+\half}) - z_{k+1}}{\eta}
                \end{aligned} 
            \end{equation}
        \ENDFOR
    \end{algorithmic}
\end{algorithm*}

We study \CMI under \Cref{assumption:CMI} with $\rho$-comonotone operators. The extra anchored gradient (EAG) algorithm~\citep{yoon2021accelerated} has been shown to have optimal last-iterate convergence rate and point convergence~\citep{yoon2022accelerated}. However, these results only hold when (i) $A = 0$, i.e., the \emph{unconstrained} setting and (ii) $F$ is \emph{monotone ($\rho \ge 0)$}. The state-of-the-art work of \citet{InIterate23} achieves similar results in the constrained setting, i.e., $A=\partial\indSet_S$ for the feasible set $S$, but still requires the monotonicity assumption on $F$. Thus the problem of achieving $O(1/T)$ convergence rate and point convergence without the monotonicity assumption is open even in the unconstrained case. In this section, we improve both results from \citep{yoon2022accelerated, InIterate23} by proposing the composite extra anchored gradient algorithm \ref{composite-EAG} that achieves both the optimal $O(1/T)$ last-iterate convergence rate (\Cref{thm: composite-EAG last-iterate rate}) and point convergence (\Cref{thm: composite-EAG merging path}) for \CMI with $\rho \ge -\frac{1}{20L}$. 
\paragraph{composite-EAG} 
We propose \emph{composite extra anchored gradient (\ref{composite-EAG})} algorithm (\Cref{alg:composite-EAG}). Specifically, given an arbitrary $z_0$ and step size $\eta > 0$, \ref{composite-EAG} first set $z_1 = J_{\eta A}[z_0 - \eta F(z_0)]$ and $c_1 = \frac{z_0 - \eta F(z_0) - z_1}{\eta}$; then for any $k \ge 1$, it updates as shown in \Cref{alg:composite-EAG}.
Recall that by definition of the resolvent $J_{\eta A}$, we have $c_k \in A(z_k)$. All missing proofs in this section can be found in \Cref{app: composite-EAG}.

\begin{remark}
    A keen reader may notice that in the important special case of constrained variational inequality problem where $A = N_\Z$ is the normal cone of a convex set $\Z \subseteq \R^n$, the $z_{k+\half}$ produced by \ref{composite-EAG} may not lie in the set $\Z$. In \Cref{sec:proj-EAG}, we present \ref{proj-EAG}, a variant of \ref{composite-EAG}, designed to consistently produce iterates within the set $\Z$. Furthermore, \ref{proj-EAG} preserves the $O(1/T)$ convergence rate alongside point convergence.
\end{remark}

\subsection{Technical Novelty}
We illustrate the main technical challenges and techniques developed in extending the algorithms and their convergence guarantees to the composite and comonotone settings, which may also be applicable to analyzing other algorithms.

\paragraph{Constraints.} One key challenge is to design the correct update rule of \emph{composite-EAG} in the composite setting. Standard approaches, such as adding two backward steps (i.e.,  projections) to EAG, do not yield algorithms that admit  simple analysis. In \ref{composite-EAG}, we introduce a new element $c_k \in A(z_k)$ to ensure that  $F(z_k) + c_k \in (F + A)(z_k)$. If we consider \(F + A\) as a single operator, the update rule of \ref{composite-EAG} simplifies to that of EAG in the unconstrained setting. 
This property enables us to smoothly extend the analysis from the unconstrained setting to the composite setting. The design of \ref{composite-EAG} also appears in the concurrent work of~\citet {kovalev2022first}. Our approach to handling constraints has inspired subsequent work by \citet{InIterate23}, which extends fast OGDA from the unconstrained to the constrained setting. 

\paragraph{Negative Comonotonicity.} Negative comonotonicity introduces an additional quadratic term in the analysis of \ref{composite-EAG}, and it is unclear how to cancel this term. We address it through a weighted sum of terms, employing Young’s inequality and a carefully balanced set of coefficients. Although this approach enables our analysis, it results in a more restrictive range of \(\rho \ge -\frac{1}{20L}\) compared to \ref{composite-FEG}.

\subsection{$O(1/T)$ Last-Iterate Convergence Rate}
In this section, we show that \ref{composite-EAG} has $O(1/T)$ last-iterate convergence rate with respect to the tangent residual, which matches the lower bound for any first-order methods~\citep{diakonikolas2020halpern, yoon2021accelerated}. 

\begin{theorem}[Last-iterate convergence rate of \ref{composite-EAG}]\label{thm: composite-EAG last-iterate rate}
Consider a comonotone inclusion problem with \Cref{assumption:CMI} holds with $\rho \ge -\frac{1}{20L}$. Let $\eta = \frac{0.31}{L}$ and $\{z_k, z_{k+\half}\}_{k \ge 1}$ the iterates generated by \ref{composite-EAG}. Then the following holds for any $T \ge 1$, 
\[
r^{tan}_{F,A}(z_T)^2 \le \InNorms{F(z_T) + c_T}^2 \le \frac{20H^2}{\eta^2 T^2},
\]     
where $H^2 := 6\InNorms{z_1 - z_0}^2 + \InNorms{z_0 - z^*}^2 \le 6\eta^2 r^{tan}_{F, A}(z_0)^2 + \InNorms{z_0 - z^*}^2$.
\end{theorem}

\paragraph{Proof Sketch} We apply a potential function argument. Recall the definition of the tangent residual $r^{tan}_{F,A}(z):= \min_{c \in A(z)} \InNorms{F(z) + c}$, which involves an optimization problem and can be hard to analyze directly. Since $c_k \in A(z_k)$, we have $\InNorms{F(z_k) + c_k} \ge r^{tan}_{F,A}(z_k)$ and $\InNorms{F(z_k) + c_k}$ can be used as a proxy for the tangent residual. We construct a potential function $V_k$ that is of the order $\Omega(k^2 \cdot \Ham_{F,A}(z_k)^2)$ (\Cref{lemma:composite-EAG lower bound V_k}). Although the potential function may not always decrease, we manage to prove that the increment in each iteration $k$ is sufficiently small: $V_{k+1} - V_k \le O(1)\cdot\InNorms{F(z_k) + c_k}^2$ in \Cref{thm:composite-EAG monotone potential}. We then apply this \emph{approximately non-increasing} property of the potential function to conclude $O(1/T)$ last-iterate convergence rate of $\InNorms{F(z_k) + c_k}$ and thus $r^{tan}_{F,A}(z_k)$. 

\subsubsection{Proof of \Cref{thm: composite-EAG last-iterate rate}}
\paragraph{Potential function} We formally define the potential function $V_k$ for $k \ge 1$ as $V_k = \frac{k(k+1)}{2} \InNorms{\eta F(z_k) + \eta c_k}^2 + k \InAngles{\eta F(z_k) + \eta c_k, z_k - z_0}$.
A bound on $V_1$ is immediate. 

\begin{lemma}[Upper bound of $V_1$]\label{lemma:composite-EAG V_1 bound}
    In the same setup as \Cref{thm: composite-EAG last-iterate rate}, we have $V_1 \le  6\InNorms{z_1 - z_0}^2 \le 6 \eta^2 r^{tan}_{F, A}(z_0)^2$ 
\end{lemma}

The following result that proves that $V_k$ is approximately non-increasing is the core of the analysis.
\begin{theorem}[Approximately non-increasing potential]
\label{thm:composite-EAG monotone potential}
    In the same setup as \Cref{thm: composite-EAG last-iterate rate}, we have for all $k \ge 1$, 
    \begin{align*}
        &V_{k+1} - V_k  \\
        &\le \frac{1}{2}\InNorms{\eta F(z_{k+1}) + \eta c_{k+1}}^2 - \frac{9k(k+1)}{4000}\InNorms{z_{k+\half} - z_{k+1}}^2.
    \end{align*}
\end{theorem}

The following lemma shows that $V_k$ is of order $\Omega(k^2 \cdot r^{tan}(z_k)^2)$. 
\begin{lemma}\label{lemma:composite-EAG lower bound V_k}
    In the same setup as \Cref{thm: composite-EAG last-iterate rate}, for any $k \ge 1$, we have
    \[
    \frac{k(k+1)}{4} \InNorms{\eta F(z_k) + \eta c_k}^2 \le V_k + \InNorms{z_0 - z^*}^2
    \]
    In particular, $V_k \ge - \InNorms{z_0 -z^*}^2$.
\end{lemma}

\paragraph{Proof of \Cref{thm: composite-EAG last-iterate rate}} Recall that $H^2 \ge V_1 + \InNorms{z_0 - z^*}^2$. For $k = 1$, from \Cref{lemma:composite-EAG lower bound V_k}, we directly get that $\InNorms{\eta F(z_1) + \eta c_1}^2 \le 4(V_1 + \InNorms{z_0 - z^*}^2)$.
    
Now fix any $k \ge 2$. Combining \Cref{thm:composite-EAG monotone potential} and \Cref{lemma:composite-EAG lower bound V_k}, we have 
\begin{align*}
    &\frac{k(k+1)}{4}\InNorms{\eta F(z_k) + \eta c_k}^2 \\
    &\le V_k + \InNorms{z_0 - z^*}^2\\
    &\le V_1 +  \InNorms{z_0 - z^*}^2 + \frac{1}{2} \cdot \sum_{t=2}^k \InNorms{\eta F(z_t) + \eta c_t}^2.
\end{align*}
Subtracting $\frac{1}{2}\InNorms{F(z_k) + \eta c_k}^2$ from both sides and noting that $\frac{1}{4}k \ge \frac{1}{2}$ for $k \ge 2$ gives
\begin{align*}
    &\frac{k^2}{4} \InNorms{\eta F(z_k) + \eta c_k}^2 \\
    &\le V_1 +  \InNorms{z_0 - z^*}^2 + \frac{1}{2} \cdot \sum_{t=2}^{k-1} \InNorms{\eta F(z_t) + \eta c_t}^2.
\end{align*}
Now we can apply \Cref{prop:1/2 sequence analysis} with $C_1 = V_1 + \InNorms{z_0 - z^*}^2$ to conclude that $\InNorms{\eta F(z_k) + \eta c_k}^2 \le \frac{20(V_1 + \InNorms{z_0 - z^*}^2)}{k^2} \le \frac{20H^2}{k^2}$. 

\begin{corollary}\label{corollary:composite-EAG bounded summation}
    In the same setup as \Cref{thm:composite-EAG monotone potential}, we have
    \[
     \sum_{k=1}^\infty k(k+1) \InNorms{z_{k+1} - z_{k+\half}}^2 \le 4000 H^2.
    \]
    It also implies $\InNorms{z_{k+1} - z_{k+\half}}^2 \le \frac{8000 H^2}{(k+1)^2}$ for all $k \ge 1$.
\end{corollary}

\subsection{ Point Convergence of \ref{composite-EAG}}
To further show point convergence of \ref{composite-EAG}, we identify a \emph{merging path (MP)} property, a notion proposed in~\citep{yoon2022accelerated}, between \ref{composite-EAG} and the \emph{optimized Halpern's method \eqref{OHM}}~\citep{lieder2021convergence, kim_accelerated_2021}. We first present the definition of the merging path property and results for \ref{OHM}.

\paragraph{The Merging Path Property~\citep{yoon2022accelerated}} Let $\mathfrak{R}$ be a deterministic algorithm. We use $\mathfrak{R}(x_0;\+P)=\{x_0, x_1, \ldots\}$ to denote that $\mathfrak{R}$ applied to problem $\+P$ with initial point $x_0$ produces iterates $x_0, x_1, \ldots$. We say two algorithms $\mathfrak{R}_1$ and $\mathfrak{R}_1$ have \emph{$O(r(k))$-merging path} if for any problem $\+P$ and point $x_0 \in \R^n$, the iterates $\mathfrak{R}_\ell(x_0; \+P) = (x_0, x_1^\ell, x_2^\ell, \ldots)$ for $\ell= 1, 2$ satisfy $\InNorms{x_k^1 - x_k^2}^2 = O(r(k))$. We say $\mathfrak{R}_1$ and $\mathfrak{R}_1$  are $O(r(k))$-MP if they have $O(r(k))$-merging paths. The merging path property is a strong notion that quantifies the near-equivalence of algorithms. \citet{yoon2022accelerated} establishes $O(1/k^2)$-MP for EAG, FEG, and the optimized Halpern's method \eqref{OHM} in the unconstrained ($A = 0$) and monotone ($\rho \ge 0$) setting. Since it is well-known that \ref{OHM} has point convergence, the merging path property implies point convergence of EAG and FEG in the unconstrained and monotone setting. 

\paragraph{Optimized Halpern's Method} For the inclusion problem of $F + A$, \ref{OHM} is an implicit method: let $w_0 = z_0$, it updates in each iteration $k \ge 0$:
\begin{equation}\label{OHM}\tag{OHM}
    \begin{aligned}
        w_{k+\half} &= \beta_k w_0 + (1- \beta_k) w_k \\
     w_{k+1} &= J_{\eta (F + A)}[w_{k+\half}]
    \end{aligned} 
\end{equation}
where we denote $\beta_k = \frac{1}{k+1}$. For the ease of analysis, we define $d_{k+1} = \frac{w_{k+\half} - \eta F(w_{k+1}) - w_{k+1}}{\eta}$. By definition of the resolvent $J_{\eta (F + A)}$, we have $d_{k+1} \in A(w_{k+1})$.  We note that when $F+A$ is maximally $\rho$-comonotone with $\rho > - \frac{1}{2L}$, its resolvent $J_{\eta (F + A)}$ is nonexpansive~\citep{bauschke2021generalized}. Thus, classical results on Halpern's iteration show that \ref{OHM} converges~\citep[Theorem 30.1]{bauschke2017convex} and has optimal last-iterate convergence rate~\citep{lieder2021convergence}. We summarize these results below.

\begin{theorem}[Theorem 30.1 of \citep{bauschke2017convex},  Theorem 2.1 of \citep{lieder2021convergence}]\label{thm:OHM} Let $w_0 = z_0 \in \Z$ be any point. If $J_{\eta(F + A)}$ has a fixed point,\footnote{$J_{\eta(F + A)}$ has a 
fixed point
is equivalent to $\Zeros(F+A)$ is nonempty.} then the iterates $\{w_k\}_{k\ge 0}$ of \ref{OHM} converge to $\Pi_{\Zeros(F + A)}[z_0] = \argmin_{z \in \Zeros(F + A)}\InNorms{z - z_0}$. Moreover, for any $k \ge 1$, it holds that 
    \begin{align*}
        \InNorms{\eta F(w_{k+1}) + \eta d_{k+1}}^2 &=  \InNorms{w_{k+\half} - J_{\eta (F + N_\Z)}[w_{k+\half}]}^2 \\
        &\le \frac{4 \InNorms{w_0 - w^*}^2}{(k+1)^2}.
    \end{align*}
\end{theorem}

\paragraph{Merging Path between \ref{composite-EAG} and \ref{OHM}} The main result of this section is to establish that \ref{composite-EAG} and \ref{OHM} are $O(\frac{1}{\sqrt{k}})$-MP, i.e., $\InNorms{z_{k+1} - w_{k+1}}^2 = O(\frac{1}{\sqrt{k}})$ (\Cref{thm: composite-EAG merging path}). 
Combining this result with the convergence guarantee of \ref{OHM} (\Cref{thm:OHM}) directly implies that \ref{composite-EAG} converges to the solution closest to the initial point $z_0$. 

\begin{theorem}[Merging path between \ref{composite-EAG} and \ref{OHM}]
\label{thm: composite-EAG merging path}
    Consider a comonotone inclusion problem with \Cref{assumption:CMI} holds with $\rho \ge -\frac{1}{20L}$. Let $\eta = \frac{0.31}{L}$. Denote $\{z_k\}_{k\ge 0}$ the iterates of \ref{composite-EAG} and $\{w_k\}_{k \ge 0}$ the iterates of \ref{OHM}. Let $w_0 = z_0$, then for any $k \ge 1$, 
    \[
    \InNorms{z_{k+1} - w_{k+1}}^2 \le \frac{40000 H^2}{(k+1)^{\half}} + \frac{30000H^2}{(k+1)^2}.
    \]
\end{theorem}
The rate of $O(\frac{1}{\sqrt{k}})$-MP is worse than $O(1/k^2)$-MP obtained by \citet{yoon2022accelerated} in the unconstrained and monotone setting. This is due to the need to tackle additional technical challenges introduced by the constraints and negative comonotonicity. Understanding the precise dependence on $k$ for the merging path property is an interesting open problem. To our knowledge, \ref{composite-EAG} is the first explicit algorithm with both optimal convergence rate and point convergence guarantee for \CMI. Indeed, such guarantees were not known even in the unconstrained setting. 

\paragraph{Proof Sketch} Fix any $k$, we conduct a single-step analysis and prove the following inequality
\begin{align*}
    (k+1)^2\InNorms{z_{k+1}-w_{k+1}}^2 \le k^2\InNorms{z_k - w_k}^2 + (k+1)^2 E_k
\end{align*}
where $E_k$ represents certain terms. By telescoping the above inequality, it remains to show that $\sum_{t=1}^k (k+1)^2 E_k = O(k^{3/2})$, which follows from combining the $O(1/k)$ last-iterate convergence rate (\Cref{thm: composite-EAG last-iterate rate}) and \Cref{corollary:composite-EAG bounded summation}.

\section{\ref{composite-FEG}: Optimal Convergence Rate for a Wider Range of $\rho$}\label{sec:composite-FEG}
The \ref{composite-EAG} algorithm has optimal last-iterate convergence rate as well as guaranteed point convergence, but only for \CMI with $\rho \ge -\frac{1}{20L}$. In this section, we propose the composite Fast Extra Gradient \eqref{composite-FEG} method, an explicit method that is applicable to \CMI with $\rho > - \frac{1}{2L}$. We show that \ref{composite-FEG} achieves optimal $O(1/T)$ last-iterate convergence rate. We remark that \CMI with $\rho > -\frac{1}{2L}$ is currently the largest class of non-monotone problems efficiently solvable by any single-loop algorithms. Our result generalizes previous results~\citep{yoon2021accelerated, lee2021fast}, which are limited to the unconstrained setting.  Additionally, we show that \ref{composite-FEG} has point convergence in the monotone setting ($\rho = 0$), which matches the state-of-the-art result by \citep{InIterate23}.

\paragraph{Extension of FEG for Comonotone Inclusion} 
Given any initial point $z_0 \in \R^n$ and step size $\eta >0$, \ref{composite-FEG} sets $c_0 = 0$ and updates $\{z_{k+\half}, z_{k+1}, c_{k+1}\}_{k\ge 0}$ as shown in \Cref{alg:composite-FEG}. Note that by definition, we have $c_k \in A(z_k)$ for all $k \ge 1$. Our algorithm is inspired by FEG~\cite{lee2021fast}. In particular, when $A(z)= 0$, i.e., the unconstrained setting, $c_k$ is always $\boldzero$, and our algorithm is identical to FEG.
\begin{algorithm*}[!ht]
    \caption{\ref{composite-FEG} for \CMI with $\rho > -\frac{1}{2L}$}
    \begin{algorithmic} \label{alg:composite-FEG}
        \STATE \textbf{Initialization:}  $z^0 \in \R^n$, $c_0 = 0$, $\beta_k = \frac{1}{k+1}$ for $k \ge 0$
        \FOR{$k \ge 0$}
        \STATE \begin{equation}\label{composite-FEG}\tag{composite-FEG}
            \begin{aligned}
                &z_{k+\half} = \beta_k z_0 + (1-\beta_k)z_k - (1-\beta_k)(\eta + 2\rho) (F(z_k) + c_k) \\
                &z_{k+1} = J_{\eta A} \InBrackets{\beta_k z_0 + (1-\beta_k)z_k -  \eta F(z_{k+\half}) - 2(1-\beta_k)\rho (F(z_k)+c_k) }\\
                &c_{k+1} = \frac{\beta_k z_0 + (1-\beta_k)z_k -  \eta F(z_{k+\half}) - 2(1-\beta_k)\rho (F(z_k)+c_k)- z_{k+1}}{\eta}
            \end{aligned}
        \end{equation}
        \ENDFOR
    \end{algorithmic}
\end{algorithm*}

Our main results of the section shows that \ref{composite-FEG} has optimal $O(\frac{1}{T})$ for the range of comonotone inclusion problems with $\rho > -\frac{1}{2L}$. Moreover, \ref{composite-FEG} enjoys point convergence in the monotone setting. All missing proofs in this section can be found in \Cref{app:composite FEG}.

\subsection{Last-Iterate Convergence Rate of \ref{composite-FEG}}
\begin{theorem}\label{thm: composite-FEG last-iterate rate}
Suppose Assumption~\ref{assumption:CMI} holds for some $\rho \in (-\frac{1}{2L}, 0]$. Let $z_0 \in \R^n$ be any starting point and $\{z_k, z_{k+\half}\}_{k\ge 1}$ be the iterates of \ref{composite-FEG} with step size $\eta \in (max(0,-2\rho),\frac{1}{L}]$. Denote $H_0^2 = \min_{z^* \in \Zeros(F+A)}\InNorms{z_0  -z^*}^2$.
Then for any $T \ge 1$, 
\begin{align*}
    \Ham_{F,A}(z_T)^2  = \min_{c \in A(z_T)}\InNorms{F(z_T) + c}^2 \le \frac{4}{(\eta + 2\rho)^2 L^2} \frac{H_0^2 L^2}{T^2}.
\end{align*}
\end{theorem}
\begin{remark}
To interpret the convergence rate, one can think of a properly selected $\eta$ such that $(\eta + 2\rho) L$ is an absolute constant, and $\Ham_{F,A}(z_T)^2$ will be $O\InParentheses{\frac{\InNorms{z_0-z^*}^2 L^2}{T^2}}$. In particular, when $\rho = 0$ (the monotone case), we can choose $\eta = \frac{1}{L}$ and get $\Ham_{F,A}(z_T)^2 \le \frac{4\InNorms{z_0 - z^*}^2 L^2}{T^2}$.
\end{remark}

We show that \ref{composite-FEG} enjoys the optimal convergence rate of $O(\frac{1}{T})$ via a potential function argument.
\paragraph{Potential function} To analyze our algorithm, we adopt the following potential function: for $k \ge 1$, 
\begin{align*}
    U_k &:= \InParentheses{ \frac{k^2}{2}\InParentheses{1+\frac{2\rho}{\eta}} - \frac{\rho}{\eta} k} \cdot \InNorms{\eta F(z_k) + \eta c_k}^2 \\
    &+ k \cdot \InAngles{\eta F(z_k) + \eta c_k, z_k - z_0}.
\end{align*}

The potential function is the same as the one used in the analysis for FEG~\cite{lee2021fast} when $c_k$ is always $0$, and we adapt it for non-zero $c_k$'s.
The analysis builds on the following two properties of the potential function:  Proposition~\ref{prop:U1 bound} establishes an upper bound of  $U_1$; Lemma~\ref{lem:composite-FEG monotone potential} shows $U_{k+1} \le U_k$ for all $k \ge 1$.
\begin{proposition}\label{prop:U1 bound}
$U_1 \le \frac{\eta^2L^2-1}{2}\InNorms{z_1-z_0}^2$. In particular, if $\eta \le \frac{1}{L}$, then $U_1\le 0$.
\end{proposition}

\begin{lemma}\label{lem:composite-FEG monotone potential}
Assume Assumption~\ref{assumption:CMI} holds for 
some $\rho$. Let $z_0 \in \R^n$ be any initial point and  $\{z_k, z_{k+\half}\}_{k\ge 1}$ be the iterates of \ref{composite-FEG} with step size 
$\eta>-2\rho$.\footnote{Lemma~\ref{lem:composite-FEG monotone potential} holds for all step size $\eta$, but our potential function is no longer useful when $\eta\leq -2\rho$.}
 Then for all $k \ge 1$, we have 
 \[
 U_{k+1} \le U_k - \frac{(1-\eta^2 L^2)(k+1)^2}{2}\InNorms{z_{k+1} - z_{k+\half}}^2.
 \]
\end{lemma}

The following lemma shows that $U_t$ is of order $\Omega(t^2 \cdot r_{F,A}^{tan}(z_t)^2)$. 
\begin{lemma}\label{lemma:composite-FEG lower bound V_k}
    In the same setup as \Cref{thm: composite-FEG last-iterate rate}, for any $t \ge 1$ and any $z^* \in \Zeros(F+A)$, we have
    \[
    \frac{\eta \InParentheses{\eta + 2\rho}t^2}{4} \cdot \InNorms{F(z_t) + c_t}^2 \le U_t + \frac{\eta}{\eta + 2\rho} \InNorms{z_0 - z^*}^2
    \]
    In particular, $U_t \ge - \InNorms{z_0 -z^*}^2$ if $\rho = 0$.
\end{lemma}

Now we are ready to prove $O(1/T)$ last-iterate convergence rate of \ref{composite-FEG}.
\paragraph{Proof of \Cref{thm: composite-FEG last-iterate rate}}
\notshow{
It is clear that $\InNorms{F(z_1) + c_1} \le \frac{1+\eta L}{\eta} \InNorms{z_1 - z_0}$ (See proof in Proposition~\ref{prop:U1 bound}). Since $c_1 \in A(z_1)$, we have 
\begin{align*}
    \min_{c \in A(z_1)}\InNorms{F(z_1) + c}^2 = \Ham_{F,A}(z_1)^2 \le \frac{(1+\eta L)^2}{\eta^2} \InNorms{z_1 - z_0}^2.
\end{align*}
Thus the claim holds for $T = 1$. 
}
By Proposition~\ref{prop:U1 bound} and the assumption that $\eta L < 1$, $U_1 \le \frac{\eta^2L^2-1}{2}\InNorms{z_1 - z_0}^2 \le 0$. Now combining \Cref{lem:composite-FEG monotone potential} and \Cref{lemma:composite-FEG lower bound V_k} gives
\begin{align*}
    &\InNorms{F(z_T) + c_T}^2 \\
    &\le \frac{4}{\eta (\eta + 2\rho)T^2} \InParentheses{ U_T + \frac{\eta}{\eta + 2\rho} \InNorms{z_0 - z^*}^2 } \\
    &\le \frac{4}{(\eta +2\rho)^2T^2} \InParentheses{ \frac{\eta + 2\rho}{\eta}U_1 + \InNorms{z_0 - z^*}^2 } \\
    & \le \frac{4}{(\eta +2\rho)^2T^2} \InNorms{z_0 - z^*}^2 \tag{$\rho \le 0$}. 
\end{align*}
This completes the proof since the above holds for all $T$ and any solution $z^*$ of the inclusion problem.

\subsection{Point Convergence of \ref{composite-FEG}}

In this section, we show that the sequence of iterates of \ref{composite-FEG} converges when $F$ is monotone, i.e., $\rho = 0$. We adopt the same idea as in the previous section. Specifically, we prove the merging path (MP) property between \ref{composite-FEG} and the \emph{optimized Halpern's method \eqref{OHM}}. We remark that MP property between FEG and OHM was established only in the unconstrained setting~\citep{yoon2022accelerated}.

\paragraph{Merging Path between \ref{composite-FEG} and \ref{OHM}} The main result of this section is to establish that \ref{composite-FEG} and \ref{OHM} are $O(\frac{1}{\sqrt{k}})$-MP, i.e., $\InNorms{z_{k+1} - w_{k+1}} = O(\frac{1}{\sqrt{k}})$.  Combining this result with the convergence guarantee of \ref{OHM} (\Cref{thm:OHM}) directly proves that \ref{composite-EAG} converges to the solution closest to the initial point $z_0$. 
\begin{theorem}[Merging path between \ref{composite-FEG} and \ref{OHM}]
\label{thm: composite-FEG merging path}
    Suppose Assumption~\ref{assumption:CMI} holds for $\rho = 0$. Let $z_0 = w_0$ and $\eta \in (0, \frac{1}{L})$. Denote $\{z_k\}_{k\ge 0}$ the iterates of \ref{composite-FEG} and $\{w_k\}_{k \ge 0}$ the iterates of \ref{OHM}.  Denote $H_0^2 = \min_{z^* \in \Zeros(F+A)}\InNorms{z_0 - z^*}^2$. Then for any $k \ge 1$, 
    \[
    \InNorms{z_{k+1} - w_{k+1}}^2 \le \frac{H_0^2}{1-\eta^2 L^2} \cdot\InParentheses{\frac{24}{(k+1)^{\half}} + \frac{36}{(k+1)^2}}.
    \]
\end{theorem}
\section{Conclusion and Discussion}
In this work, we propose two novel algorithms for solving comonotone inclusion problems with provable convergence guarantees. 
The \ref{composite-EAG} algorithm recovers EAG from \citep{yoon2021accelerated} in the unconstrained case and has $O(1/T)$ last-iterate convergence rate and the favorable point convergence for $\rho$-comonotone with $\rho \ge -\frac{1}{20L}$. The \ref{composite-FEG} algorithm recovers FEG from \citep{lee2021fast} in the unconstrained case, and has $O(1/T)$ last-iterate convergence rate for $\rho$-comonotone problems with $\rho > -\frac{1}{2L}$ while maintains point convergence in the monotone case. An interesting open problem is to design an algorithm that obtains both optimal convergence rate and point convergence for a wider range of $\rho$-comonotone problems.

\section*{Impact Statement}
This paper presents work that aims to advance the field of Machine Learning. Our work has many potential societal consequences, none of which we feel must be specifically highlighted here.

\section*{Acknowledgment}
We thank the anonymous ICML reviewers for their constructive comments, which improved the presentation of the work. WZ thanks Ruicheng Jiang for an observation, which leads to an improved constant in \Cref{thm: composite-FEG last-iterate rate}. YC and WZ are supported by the NSF Awards CCF-1942583 (CAREER) and CCF-2342642. AO is supported by NSF Award CCF-1942583 (CAREER) and a Meta Ph.D. Research Fellowship. WZ also receives a research fellowship from the Center for Algorithms, Data, and Market Design at Yale (CADMY).
 
\bibliography{ref, references}
\bibliographystyle{icml2024}

\newpage
\appendix
\onecolumn
\appendix

\section{Additional Related Work}\label{apx:related work}
\subsection{Additional Classical Results for Convex-Concave Min-Max Optimization}
\paragraph{Convergence in Gap Function.} Nemirovski and Nesterov show that the average iterate of extragradient-type methods has $O(\frac{1}{T})$ convergence rate in terms of gap function defined as $\max_{z' \in \Z}\InAngles{F(z'), z - z'}$~\citep{nemirovski_prox-method_2004,nesterov_dual_2007}, matching the lower bound for first-order methods \citep{ouyang2021lower}. Their results extend to the more general monotone variational inequalities. Convergence to an approximate solution of \CMI implies convergences with the same rate for the gap function, while the converse is not true (see \Cref{ex:weak gap}). 

\paragraph{Convergence of the Extragradient Method in Stronger Performance Measures.} For stronger performance measures such as the norm of the operator (when $\Z = \R^n$) or the residual (in constrained setting), classical results show that the best-iterate of the extragradient method converges at a rate of $O(\frac{1}{\sqrt{T}})$~\citep{korpelevich_extragradient_1976,facchinei_finite-dimensional_2003}. Recently, the same convergence rate is shown to hold even for the last-iterate of the extragradient method~\citep{gorbunov2022extragradient,cai2022finite}. Although $O(\frac{1}{\sqrt{T}})$ convergence on the residual is optimal for all $p$-SCIL algorithms \citep{golowich2020last, golowich2020tight}, a subclass of first-order methods that includes the extragradient method and many of its variations, faster rate is possible for other first-order methods. 

\subsection{Additional Results on Nonconvex-Nonconcave Min-Max Optimization}

\paragraph{(Weak) Minty Variational Inequality} We only introduce the definitions in the unconstrained setting, as that is the setting considered by several recent results, and all convergence rates are with respect to the operator norm. The \emph{Minty variational inequality} (MVI) condition (also called coherence or variationally stable) assumes the existence of point $z^*\in \real^n$ such that \begin{align*}
    \InAngles{F(z), z - z^*} \ge 0,\quad \forall z \in \R^n
\end{align*} 
is studied in e.g., \cite{dang2015convergence,zhou2017stochastic,liu2019towards,malitsky2020golden,song2020optimistic,liu2021first}. Extragradient-type algorithms has $O(\frac{1}{\sqrt{T}})$ convergence rate for Lipschitz operators that satisfy the MVI condition \citep{dang2015convergence}. \citet{diakonikolas2021efficient} proposes a weaker condition called \emph{weak MVI}: there exits $z^*$ and $\rho < 0$ such that
\begin{align*}
    \InAngles{F(z), z - z^*} \ge \rho \cdot \InNorms{F(z)}^2, \quad \forall z \in \R^n.
\end{align*} 
The weak MVI condition includes both MVI and negative comonotonicity \citep{bauschke2021generalized} as special cases. \citet{diakonikolas2021efficient} proposes the EG+ algorithm, which has $O(\frac{1}{\sqrt{T}})$ convergence rate under the weak MVI condition in the unconstrained setting. \citet{bohm2022solving} propose OGDA+ in the same setting while \citet{pethick2022escaping} generalizes EG+ to CEG+ algorithm which has $O(\frac{1}{\sqrt{T}})$ convergence rate under weak MVI condition in general (constrained) setting. 

\section{Additional Preliminaries}
\subsection{Monotone Inclusion and Variational Inequality}\label{sec:preliminary-MI&VI}

\paragraph{Variational Inequality.} An important special case of \CMI is the \emph{variational inequality} (VI) problem where the maximally monotone operator $A$ is chosen to be a normal cone operator $N_\Z$ for a closed and convex feasible set $\Z$. The VI problem has two variants. The \emph{Stampacchia Variational Inequality} (SVI) problem is to find $z^* \in \Z$ such that 
\begin{equation}\label{SVI}\tag{SVI}
    \InAngles{F(z^*), z^* - z} \le 0, \quad \forall z \in \Z.
\end{equation}
Such $z^*$ is called a \emph{strong solution} to VI. The \emph{Minty Variational Inequality} (MVI) problem is to find $z^* \in \Z$ such that  
\begin{equation}\label{MVI}\tag{MVI}
    \InAngles{F(z), z^* - z} \le 0, \quad \forall z \in \Z.
\end{equation}
Such $z^*$ is called a \emph{weak solution} to VI. 

When $F$ is continuous, then every solution to \ref{MVI} is also a solution to \ref{SVI}. When $F$ is monotone, every solution to \ref{SVI} is also a solution to (\ref{MVI}), and thus, the two solution sets are equivalent. We remark that the solution set to \CMI is the same as the solution set to \ref{SVI}.
For a convex set $\Z$, $\partial \indSet_\Z$ is a maximally monotone set-valued operator, and as such, \CMI generilizes MVI, and convex-concave min-max optimization problems for ($\rho=0$ and $A=\partial\indSet_\Z$) and serves as computational frameworks for numerous important applications in fields such as economics, engineering, probability and statistics, and machine learning~\cite{facchinei_finite-dimensional_2003,bauschke_convex_2011,ryu2016primer}.

Although the set of exact solutions to \CMI for $A=\partial \indSet_\Z$ coincides with the set of exact solutions to the corresponding variational inequality, the approximate solutions to these two problems differ due to different performance measures. An approximate solution to the comonotone inclusion problem~\CMI must have a small natural residual,\footnote{The natural residual of a point $z$ is simply the operator norm $\InNorms{F(z)}$ in the unconstrained case, i.e., $\Z=\real^n$, and equals to the norm of its natural map $z-\Pi_\Z[z-F(z)]$~\cite{facchinei_finite-dimensional_2003}.} while an approximate solution to the variational inequality only satisfies a weaker condition, i.e., its gap function is small.\footnote{There are several variations of the gap function. Depending on the exact definition, a small gap function value could mean an approximate \emph{weak} solution, i.e., approximately solve the Minty Variational Inequality (MVI), or an approximate \emph{strong} solution, i.e., approximately solve the Stampacchia Variational Inequality (SVI). Formal definitions and discussions are in Section~\ref{sec:preliminary-MI&VI}.} Indeed, it is well-known that an approximate solution to \CMI is also an approximate solution to the monotone variational equality, but the reverse is not true in general.

\paragraph{Approximate Solutions.}
We say $z \in \Z$ is an $\epsilon$-approximate solution to \ref{SVI} or \ref{MVI} if 
\begin{align*}
    &\InAngles{F(z), z - z'} \le \epsilon,  \forall z' \in \Z, \text{   or   } \\ &\InAngles{F(z'), z - z'} \le \epsilon,  \forall z' \in \Z, \text{ respectively.}
\end{align*}
When $F$ is monotone, it is clear that every $\epsilon$-approximate solution to \ref{SVI} is also an $\epsilon$-approximate solution to \ref{MVI}; but the reverse does not hold in general. 
When $F$ is monotone and $\Z$ is bounded by $D$, then any $\frac{\epsilon}{D}$-approximate solution to \CMI is an $\epsilon$-approximate solution to \ref{SVI} \citep[Fact 1]{diakonikolas2020halpern}. Note that when $\Z$ is unbounded, neither \ref{SVI} nor \ref{MVI} can be approximated. If we restrict the domain to be a bounded subset of (possibly unbounded) $\Z$, then we can define the (restricted) gap functions~\citep{nesterov_dual_2007} as
\begin{align*}
    \gap^{SVI}_{F,D}(z) := \max_{z' \in \Z \cap \mathcal{B}(z,D)} \InAngles{F(z), z - z'},\\
    \gap^{MVI}_{F,D}(z) := \max_{z' \in \Z \cap \mathcal{B}(z,D)} \InAngles{F(z'), z - z'}.
\end{align*}
The $O(\frac{1}{T})$ convergence rate for extragradient-type algorithm \citep{nemirovski_prox-method_2004,nesterov_dual_2007} is provided in terms of $\gap^{MVI}_{F,D}(z)$, which means convergence to an approximate \emph{weak} solution. Prior to our work, the $O(\frac{1}{T})$ convergence rate on $\gap^{SVI}_{F,D}(z)$ was only known in the unconstrained setting \citep{yoon2021accelerated}. When $F$ is monotone, then the tangent residual $\Ham_{F,D}(z) \le \frac{\epsilon}{D}$ (definition in section~\ref{sec:convergence criteria}) implies $\gap_{F,D}^{SVI}(z) \le \epsilon$ \citep[Lemma 2]{cai2022finite}. Therefore, our results (\Cref{thm: composite-EAG last-iterate rate}, \ref{thm: composite-FEG last-iterate rate}, and \ref{thm:proj-EAG last rate}) for the tangent residual and the natural residual also imply an $O(\frac{1}{T})$ convergence rate on $\gap_{F,D}^{SVI}(z)$ for \ref{SVI}. In the following, we show that the converse is not true in general.

\begin{example}[Gap function is weaker than natural residual]\label{ex:weak gap}
Consider an instance of the Monotone VI problem on the identity operator $F(x)=x$ in $\Z=[0,1]$.
\begin{itemize}
    \item Observe that the natural residual on $x\in \Z$ is $\InNorms{x - \Pi_{\Z}[x-F(x)]}=x$.
    \item Moreover, since $\Z=[0,1]$, observe that for any $x\in \Z$ and $D\geq 0$,
    \begin{align*}
    \gap_{F,D}^{SVI}(x)\leq \gap_{F,1}^{SVI}(x)
    =&\max_{x' \in [0,1]} x\cdot ( x - x')=x^2, \text{ and } \\
    \gap_{F,D}^{MVI}(x)\leq \gap_{F,1}^{MVI}(x)
    =&\max_{x' \in [0,1]} x'\cdot (x - x') = \frac{x^2}{4}.
    \end{align*}
\end{itemize}
As a result, any algorithm with $O(\frac{1}{T})$ convergence rate with respect to the gap function only implies a $O(\frac{1}{\sqrt{T}})$ convergence rate for the corresponding \CMI problem or the natural residual. 
\end{example}

\subsection{Proof of \Cref{fact:tangent residual smaller than natural residual}}
\label{app:ts<=ns}
\begin{proof}
For any $c \in A(z)$, we have
\begin{align*}
    r^{nat}_{F,A}(z) &= \InNorms{z - J_A\InBrackets{z - F(z)}} \\
    & = \InNorms{J_A\InBrackets{z + c}- J_A\InBrackets{z - F(z)}} \tag{$z = J_A[z+c]$} \\
    & \le \InNorms{F(z) + c}. \tag{non-expansiveness of $J_A$}
\end{align*}
Thus $r^{nat}_{F,A}(z) \le \min_{c \in A(z)} \InNorms{F(z) + c} = \Ham_{F,A}(z)$. 
\end{proof}

\section{Missing Proofs in \Cref{sec:composite-EAG}}
\label{app: composite-EAG}

\subsection{Proof of \Cref{lemma:composite-EAG V_1 bound}}
\begin{proof}
    We recall that $z_1 = J_{\eta A}[z_0 - \eta F(z_0)]$. Thus by non-expansiveness of the resolvent $J_{\eta A}$, we have for any $c_0 \in A(z_0)$, 
    \[
    \InNorms{z_1 - z_0} = \InNorms{J_{\eta A}[z_0 - \eta F(z_0)] - J_{\eta A}[z_0 - \eta c_0]} \le  \eta \InNorms{F(z_0) + c_0}. 
    \]
    Since the above holds for all $c_0 \in A(z_0)$, we get $\InNorms{z_1 - z_0} \le \eta r^{tan}_{F,A}(z_0)$. By definition of $c_1$, we also have
    \begin{align*}
        \InNorms{\eta F(z_1) + \eta c_1}^2 & = \InNorms{z_0 - z_1 + \eta F(z_1) - \eta F(z_0)}^2 \\
        &\le 2\InNorms{z_0 - z_1}^2 + 2 \InNorms{\eta F(z_1) - \eta F(z_0)}^2 \\
        &\le 4\InNorms{z_0 - z_1}^2.
    \end{align*}
    Now by definition of $V_1$, we have
    \begin{align*}
        V_1 &= \InNorms{\eta (F(z_1) + c_1)}^2 + \eta \InAngles{F(z_1) + c_1, z_1 - z_0} \\
        &\le \InNorms{\eta (F(z_1) + c_1)}^2 + \InNorms{\eta (F(z_1) + c_1)} \InNorms{z_1 - z_0} \\
        &\le 6\InNorms{z_1 - z_0}^2.
    \end{align*}
    This completes the proof.
\end{proof}

\subsection{Proof of \Cref{thm:composite-EAG monotone potential}}
\begin{proof}
    We require the following fact. 
    \begin{fact}\label{fact:step size}
        For any $L > 0$ and $\rho \ge - \frac{1}{20L}$,  $\eta = \frac{0.31}{L}$ satisfies
    \begin{align}
     1 + \frac{4\rho}{\eta} - \InParentheses{3 - \frac{4\rho}{\eta}}\cdot \eta^2 L^2 \ge \frac{9}{2000}, \label{eq:comonotone step size}
    \end{align}
    and $\frac{\rho}{\eta} > -\frac{1}{4}$
    \end{fact}
    
    Our plan is to show $V_k - V_{k+1}$ minus two non-negative terms is greater than $- \frac{1}{2}\InNorms{\eta F(z_{k+1}) + c_{k+1}}^2 + \frac{9k(k+1)}{4000}\InNorms{z_{k+\half} - z_{k+1}}^2$. 
    \paragraph{Non-negative Terms} Let $p = \frac{1}{3}$ and $c = -\frac{4p\rho}{\eta} \ge 0$. Since $F$ is $L$-Lipschitz,  we have
\begin{align*}
    \eta^2 L^2 \cdot \InNorms{z_{k+\half} - z_{k+1}}^2 - \InNorms{\eta F(z_{k+\half}) - \eta F(z_{k+1})}^2 \ge 0.
\end{align*}
Multiplying both sides of the above inequality by $(1+c)$ and rearranging terms, we get
\begin{align}
&p \cdot \InNorms{z_{k+\half} - z_{k+1}}^2 - \InNorms{\eta F(z_{k+\half}) - \eta F(z_{k+1})}^2 \nonumber \\
& + \InParentheses{(1+c)\eta^2 L^2 - p} \cdot \InNorms{z_{k+\half}}^2 - c \cdot \InNorms{\eta F(z_{k+\half}) - \eta F(z_{k+1})}^2 \ge 0. \label{eq:composite-EAG Lip}
\end{align}
Moreover, since $c_k \in A(z_k)$ and $c_{k+1} \in A(z_{k+1})$ and the fact that $F + A$ is $\rho$-comonotone, we have
\begin{align}
    \InAngles{\eta F(z_{k+1}) + \eta c_{k+1} -\eta F(z_k) - \eta c_k, z_{k+1} - z_k} - \frac{\rho}{\eta} \InNorms{\eta F(z_{k+1}) + \eta c_{k+1} -\eta F(z_k) - \eta c_k}^2 \ge 0. \label{eq:composite-EAG comonotone}
\end{align}

\paragraph{Descent Identity} We have the following identity holds by \Cref{prop:EAG identity} (\Cref{eq:identity composite-EAG}) since $\eta c_k = z_k - \eta F(z_k) + \frac{1}{k+1}(z_0 - z_k) - z_{k+\half}$ and $\eta c_{k+1} = z_k - \eta F(z_{k+\half}) + \frac{1}{k+1}(z_0 - z_k) - z_{k+1}$. 
\begin{align}
    &V_k - V_{k+1} - \frac{k(k+1)}{2p} \cdot \LHSI~\eqref{eq:composite-EAG Lip} - k(k+1) \cdot \LHSI~\eqref{eq:composite-EAG comonotone}  \nonumber \\
    &= \frac{(1-p)k(k+1)}{2p} \InNorms{\eta F(z_{k+\half}) - \eta F(z_{k+1})}^2 \label{eq:composite-EAG RHS-1} \\
    &\quad + (k+1) \cdot \InAngles{\eta F(z_{k+\half}) - \eta F(z_{k+1}), \eta F(z_{k+1}) + \eta c_{k+1}} \label{eq:composite-EAG RHS-2}\\
    &\quad + \frac{k(k+1)}{2p}\InParentheses{ (p - (1+c)\eta^2 L^2) \cdot \InNorms{z_{k+\half}- z_{k+1}}^2 +c \InNorms{\eta F(z_{k+\half}) - \eta F(z_{k+1})}^2} \label{eq:composite-EAG RHS-3}\\
    &\quad + \frac{k(k+1)\rho}{\eta} \InNorms{\eta F(z_{k+1}) + \eta c_{k+1} -\eta F(z_k) - \eta c_k}^2. \label{eq:composite-EAG RHS-4}
\end{align}
Since $\InNorms{a}^2 + \InAngles{a,b} = \InNorms{a+\frac{b}{2}}^2 -\frac{\InNorms{b}^2}{4}$, we have
\begin{align*}
    &\text{Expression} (\ref{eq:composite-EAG RHS-1}) + \text{Expression} (\ref{eq:composite-EAG RHS-2}) \\
    =& \InNorms{\sqrt{\frac{(1-p)k(k+1)}{2p}}\cdot \InParentheses{\eta F(z_{k+\half}) - \eta F(z_{k+1})} + \sqrt{\frac{p(k+1)}{2(1-p) k}} \cdot \InParentheses{\eta F(z_{k+1}) + \eta c_{k+1}}}^2 \notag \\
     & - \frac{k+1}{2k} \cdot \frac{p}{1-p} \InNorms{\eta F(z_{k+1}) + \eta c_{k+1}}^2\\ 
     & \ge -\frac{p}{1-p} \InNorms{\eta F(z_{k+1}) + \eta c_{k+1}}^2. \tag{$k \ge 1$}\\
     & = - \frac{1}{2} \InNorms{\eta F(z_{k+1}) + \eta c_{k+1}}^2.
\end{align*}
Now it remains to give a non-negative lower bound of Expression \eqref{eq:composite-EAG RHS-3} + Expression \eqref{eq:composite-EAG RHS-4}. Recall that $p = \frac{1}{3}$ and $c = -\frac{4\rho p}{\eta}$, thus 
\begin{align*}
   &(\frac{2}{k(k+1)}) \cdot \InParentheses{ \text{Expression } \eqref{eq:proj-EAG low deg RHS 4} + \text{Expression } \eqref{eq:proj-EAG low deg RHS 5}} \\
   &= \InParentheses{1 - \InParentheses{3 - \frac{4\rho}{\eta}}\cdot \eta^2 L^2}\cdot \InNorms{z_{k+\half} - z_{k+1}}^2 - \frac{4\rho}{\eta} \cdot  \cdot \InNorms{\eta F(z_{k+\half})-\eta F(z_{k+1})}^2 \\
   & \quad + \frac{2\rho}{\eta}\cdot \InNorms{\eta F(z_{k+1}) + \eta c_{k+1} -\eta F(z_k) - \eta c_k}^2   \\
   & \ge \InParentheses{1 - \InParentheses{3 - \frac{4\rho}{\eta}}\cdot \eta^2 L^2}\cdot \InNorms{z_{k+\half} - z_{k+1}}^2 + \frac{4\rho}{\eta}\cdot \InNorms{\eta F(z_{k+\half}) + \eta c_{k+1}  - \eta F(z_k) - \eta c_k}^2 
   \tag{$  \InNorms{A}^2 - \frac{1}{2}\InNorms{B}^2 \ge -\InNorms{A+B}^2$}\\
   & = \InParentheses{1 + \frac{4\rho}{\eta} - \InParentheses{3 - \frac{4\rho}{\eta}}\cdot \eta^2 L^2} \InNorms{z_{k+\half} - z_{k+1}}^2 \\
   &\ge \frac{9}{2000} \InNorms{z_{k+\half} - z_{k+1}}^2,
\end{align*}
where in the last equality we use the fact that $z_{k+\half} - z_{k+1} = \eta F(z_{k+\half}) + \eta c_{k+1}  - \eta F(z_k) - \eta c_k$ which holds by the update rule of \ref{composite-EAG}, and in the last inequality we use \Cref{fact:step size}.
\end{proof}

\subsection{Proof of \Cref{lemma:composite-EAG lower bound V_k}}
\begin{proof}
    Fix any $k \ge 1$. Since $0 \in F(z^*) + A(z^*)$, by $\rho$-comonotonicity and \Cref{fact:step size}, we have 
    \begin{equation}
        \InAngles{\eta F(z_k) + \eta c_k, z_k - z^*} \ge \frac{\rho}{\eta} \InNorms{\eta F(z_k) + \eta c_k}^2 > -\frac{1}{4} \InNorms{\eta F(z_k) + \eta c_k}^2. \label{eq:composite EAG comonotone-1}
    \end{equation}
    By definition of $V_k$, we have
    \begin{align*}
        V_k &= \frac{k(k+1)}{2}\cdot \InNorms{\eta F(z_k) + \eta c_k}^2 + k\cdot \InAngles{\eta F(z_k) + \eta c_k, z_k - z_0} \\
        &=  \frac{k(k+1)}{2}\cdot \InNorms{\eta F(z_k) + \eta c_k}^2 + k\cdot \InAngles{\eta F(z_k) + \eta c_k, z_k - z^*} + k\InAngles{\eta F(z_k) + \eta c_k, z^* - z_0} \\
        &\ge \frac{k(k+1)}{2}\cdot \InNorms{\eta F(z_k) + \eta c_k}^2 -\frac{1}{4}\InNorms{\eta F(z_k) + \eta c_k}^2 + k\InAngles{\eta F(z_k) + \eta c_k, z^* - z_0}  \tag{By \eqref{eq:composite EAG comonotone-1}} \\
        &\ge \frac{k(k+\frac{1}{2})}{2}\cdot \InNorms{\eta F(z_k) + \eta c_k}^2 - \frac{k^2}{4} \InNorms{\eta F(z_k) + \eta c_k}^2 - \InNorms{z_0 - z^*}^2 \\
        &= \frac{k(k+1)}{4}\cdot \InNorms{\eta F(z_k) + \eta c_k}^2 - \InNorms{z_0 - z^*}^2,
    \end{align*}
    where in the second last inequality, we apply $\InAngles{a, b} \ge -\frac{\alpha}{4}\InNorms{a^2} - \frac{1}{\alpha}\InNorms{b^2}$ with $a = \eta F(z_k) + \eta c_k$, $b = z^* -z_0$, and $\alpha = k$. This completes the proof.
\end{proof}

\subsection{Proof of \Cref{corollary:composite-EAG bounded summation}}
\begin{proof}
    By \Cref{thm:composite-EAG monotone potential}, we have
    \[
    \frac{9k(k+1)}{4000}\InNorms{z_{k+1} - z_{k+\half}}^2 \le V_k - V_{k+1} + \frac{1}{2}\InNorms{\eta F(z_{k+1})+\eta c_{k+1}}^2
    \]
    Telescoping the above inequality for $k = 1, 2, \ldots, T$ gives
    \begin{align*}
        \frac{9}{4000} \sum_{k=1}^T k(k+1) \InNorms{z_{k+1} - z_{k+\half}}^2 &\le V_1- V_{T+1} + \frac{1}{2} \sum_{k=1}^T \InNorms{\eta F(z_{k+1})+\eta c_{k+1}}^2 \\
        &\le V_1 + \InNorms{z_0 -z^*}^2 + 10H^2 \cdot \InParentheses{\sum_{k=1}^T \frac{1}{(t+1)^2}} \\
        & \le 9H^2.
    \end{align*}
    where in the second inequality we use (1) $V_{T+1} \ge -\InNorms{z_0 - z^*}^2$ by \Cref{lemma:composite-EAG lower bound V_k} and (2) $\InNorms{\eta F(z_{k+1})+\eta c_{k+1}}^2 \le \frac{20H^2)}{(t+1)^2}$ by \Cref{thm:proj-EAG monotone potential at constrained}, in the last inequality we use $\sum_{k=1}^\infty \frac{1}{(t+1)^2} = \frac{\pi^2}{6} -1 \le \frac{3}{4}$. This concludes $\sum_{k=1}^\infty k(k+1) \InNorms{z_{k+1} - z_{k+\half}}^2 \le 4000 H^2$. Since $\frac{1}{2}(k+1)^2 \le k(k+1)$ for all $k \ge 1$, it further implies $\InNorms{z_{k+1} - z_{k+\half}}^2 \le \frac{8000 H^2}{(k+1)^2}$ for all $k \ge 1$.
\end{proof}

\subsection{Proof of \Cref{thm: composite-EAG merging path}}
\begin{proof}
    Using the update rule of \ref{composite-EAG} and \ref{OHM} and recall: $\eta c_{k+1} := z_{k} - \eta F(z_{k+\half}) + \beta_k (z_0 - z_{k}) - z_{k+1}$ and $\eta d_{k+1} = w_{k+\half} - \eta F(w_{k+1}) - w_{k+1}$ we have
    \begin{align*}
        z_{k+1} - w_{k+1} &= \InParentheses{\beta_k z_0 + (1-\beta_k)z_k - \eta F(z_{k+\half}) - \eta c_{k+1}} - \InParentheses{\beta_k w_0 + (1-\beta_k)w_k -\eta F(w_{k+1}) - \eta d_{k+1}}\\
        & = (1-\beta_k) (z_k - w_k) + \eta \InParentheses{F(w_{k+1}) + d_{k+1} - F(z_{k+\half}) - c_{k+1}}, \tag{$w_0 = z_0$}
    \end{align*}
    which implies
    \begin{align*}
        \InNorms{z_{k+1} - w_{k+1}}^2 &= (1-\beta_k)^2 \InNorms{z_k - w_k}^2 +  \eta^2 \InNorms{F(w_{k+1}) + d_{k+1} - F(z_{k+\half}) - c_{k+1}}^2 \\
        & \quad + \underbrace{{2\InAngles{(1-\beta_k)(z_k - w_k), \eta \InParentheses{F(w_{k+1}) + d_{k+1} - F(z_{k+\half}) - c_{k+1}}}}}_{\textbf{I}}.
    \end{align*}
    We focus on term \textbf{I}.  We can verify
    \[
    z_{k+\half} - w_{k+1} = (1-\beta_k)(z_k - w_k) - \eta\InParentheses{ F(z_k) + c_{k} -  F(w_{k+1}) -  d_{k+1}}. 
    \]
    Thus term \textbf{I} can be rewritten as 
    \begin{align*}
        \textbf{I} &= 2\InAngles{(1-\beta_k)(z_k - w_k), \eta \InParentheses{F(w_{k+1}) + d_{k+1} - F(z_{k+\half}) - c_{k+1}}} \\
        &=2\InAngles{ z_{k+\half} - w_{k+1} + \eta\InParentheses{ F(z_k) + c_{k} -  F(w_{k+1}) -  d_{k+1}}, \eta \InParentheses{F(w_{k+1}) + d_{k+1} - F(z_{k+\half}) - c_{k+1}}} \\
        & = 2\eta^2 \InAngles{F(z_k) + c_k -  F(w_{k+1}) -  d_{k+1}, F(w_{k+1}) + d_{k+1} - F(z_{k+\half}) - c_{k+1}} + \\
        &\quad + \underbrace{2 \eta \InAngles{z_{k+\half} - w_{k+1},  F(w_{k+1}) + d_{k+1} - F(z_{k+\half}) - c_{k+1} } }_{\textbf{$A_k$}}.
    \end{align*}
    Combining the above (we keep term \textbf{$A_k$} for now), we get 
    \begin{align*}
        &\InNorms{z_{k+1} - w_{k+1}}^2\\
        &\le (1-\beta_k)^2 \InNorms{z_k - w_k}^2 +  \eta^2 \InNorms{F(w_{k+1}) + d_{k+1} - F(z_{k+\half}) - c_{k+1}}^2 \\
        &\quad + 2\eta^2 \InAngles{F(z_k) + c_k -  F(w_{k+1}) -  d_{k+1}, F(w_{k+1}) + d_{k+1} - F(z_{k+\half}) - c_{k+1}} + \textbf{$A_k$} \\
        &= (1-\beta_k)^2 \InNorms{z_k - w_k}^2 - \eta^2 \InNorms{F(w_{k+1}) + d_{k+1}}^2 + 2\eta^2 \InAngles{F(z_k) + c_k, F(w_{k+1}) + d_{k+1}  } \\
        &\quad - 2\eta^2 \InAngles{F(z_k) + c_k, F(z_{k+\half}) + c_{k+1}} + \eta^2 \InNorms{F(z_{k+\half}) + c_{k+1}}^2 + \textbf{$A_k$}. \\
        &\le  (1-\beta_k)^2 \InNorms{z_k - w_k}^2 + \eta^2 \InNorms{F(z_k) + c_k}^2 \tag{We use $-a^2 + 2ab \le b^2$} \\
        &\quad - 2\eta^2 \InAngles{F(z_k) + c_k, F(z_{k+\half}) + c_{k+1}} + \eta^2 \InNorms{F(z_{k+\half}) + c_{k+1}}^2 + \textbf{$A_k$} \\
        & = (1-\beta_k)^2 \InNorms{z_k - w_k}^2 + \eta^2 \InNorms{F(z_k) + c_k - F(z_{k+\half}) - c_{k+1}}^2 + \textbf{$A_k$} \\
        & = (1-\beta_k)^2 \InNorms{z_k - w_k}^2 +  \InNorms{z_{k+1} - z_{k+\half}}^2 + \textbf{$A_k$},
    \end{align*}
    where in the last equality we use the fact that $z_{k+1} - z_{k+\half} = \eta(F(z_k) + c_k - F(z_{k+\half}) - c_{k+1})$ by update rule of \ref{composite-EAG}. Plugging $\beta_k = \frac{1}{k+1}$ and multiplying both sides with $(k+1)^2$ gives
    \[
    (k+1)^2\InNorms{z_{k+1} - w_{k+1}}^2 \le k^2\InNorms{z_k - w_k}^2 + (k+1)^2 \InNorms{z_{k+1} - z_{k+\half}}^2 + (k+1)^2\textbf{$A_k$}.
    \]
    Telescoping the above gives 
    \begin{align*}
        &(k+1)^2\InNorms{z_{k+1} - w_{k+1}}^2 \le \InNorms{z_1 - w_1}^2 + \sum_{t=1}^k (t+1)^2 \InNorms{z_{t+1} - z_{t+\half}}^2 + \sum_{t=1}^k (t+1)^2 \textbf{$A_t$}. \\ 
        &\Rightarrow \InNorms{z_{k+1} - w_{k+1}}^2 \le \frac{\InNorms{z_1 - w_1}^2}{(k+1)^2} + \frac{1}{(k+1)^2}\sum_{t=1}^k (t+1)^2 \InNorms{z_{t+1} - z_{t+\half}}^2 + \frac{1}{(k+1)^2}\sum_{t=1}^k (t+1)^2 \textbf{$A_t$}
    \end{align*}
    It remains to bound $\frac{1}{(k+1)^2}\sum_{t=1}^k (t+1)^2 \InNorms{z_{t+1} - z_{t+\half}}^2$ and $\frac{1}{(k+1)^2}\sum_{t=1}^k (t+1)^2 \textbf{$A_t$}$. 

    For the first term, using \Cref{corollary:composite-EAG bounded summation}, we have
    \begin{align*}
        \sum_{t=1}^k (t+1)^2 \InNorms{z_{t+1} - z_{t+\half}}^2 &\le 2 \sum_{t=1}^k t(t+1) \InNorms{z_{t+1} - z_{t+\half}}^2 \le 8000 H^2.
    \end{align*}

    For the term with $A_t$, we need a more careful analysis.  We decompose $A_k = B_k + C_k + D_k$ as follows.
    \begin{align*}
        A_k &= 2 \eta \InAngles{z_{k+\half} - w_{k+1},  F(w_{k+1}) + d_{k+1} - F(z_{k+\half}) - c_{k+1}} \\
        &= 2 \eta \InAngles{z_{k+\half} -  z_{k+1},  F(w_{k+1}) + d_{k+1} - F(z_{k+\half}) - c_{k+1}} \\
        & \quad + 2\eta \InAngles{z_{k+1} - w_{k+1}, F(w_{k+1}) + d_{k+1} - F(z_{k+\half}) - c_{k+1}} \\
        &= \underbrace{2 \eta \InAngles{z_{k+\half} - z_{k+1},  F(w_{k+1}) + d_{k+1} - F(z_{k+\half}) - c_{k+1}}}_{B_k} + \underbrace{2\eta \InAngles{z_{k+1} - w_{k+1}, F(z_{k+1}) - F(z_{k+\half})}}_{C_k} \\
        &\quad  + \underbrace{2\eta \InAngles{ z_{k+1} - w_{k+1}, F(w_{k+1}) + d_{k+1} - F(z_{k+1}) - c_{k+1} }}_{D_k}.
    \end{align*}

     For $B_k$, we have
    \begin{align*}
        B_k &= 2\eta \InAngles{ z_{k+\half} - z_{k+1}, F(w_{k+1}) + d_{k+1} - F(z_{k+\half}) - c_{k+1} } \\
        &\le 2 \eta \InNorms{z_{k+1} - z_{k+\half}} \InParentheses{\InNorms{F(w_{k+1}) + d_{k+1}} + \InNorms{F(z_{k+1}) - F(z_{k+\half})}+\InNorms{F(z_{k+1}) + c_{k+1}}}\\
        & \le 2\eta L \InNorms{z_{k+1} - z_{k+\half}}^2 + 2 \eta \InNorms{z_{k+1} - z_{k+\half}} \InParentheses{\InNorms{F(w_{k+1}) + d_{k+1}} +\InNorms{F(z_{k+1}) + c_{k+1}}}\\
        &\le  \frac{16000 H^2}{(k+1)^2} + 2\eta \sqrt{\frac{8000 H^2}{(k+1)^2}} \cdot \InParentheses{ \frac{2\InNorms{z_0 - z^*}}{\eta (k+1)} + \frac{\sqrt{20 H^2}}{\eta(k+1)} } \\
        &\le \frac{17400 H^2}{(k+1)^2},
    \end{align*}
    where in the second last inequality we use the last-iterate convergence rate of \ref{OHM} (\Cref{thm:OHM}) and \ref{composite-EAG} (\Cref{thm: composite-EAG last-iterate rate}) as well as the bound on $\InNorms{z_{k+1} - z_{k+\half}}^2$ (\Cref{corollary:composite-EAG bounded summation}).

    For $C_k$, we have 
    \begin{align*}
        C_k &= 2\eta \InAngles{z_{k+1} - w_{k+1}, F(z_{k+1}) - F(z_{k+\half})} \\
        &\le 2\eta \InNorms{z_{k+1} - w_{k+1}} \InNorms{ F(z_{k+1}) -  F(z_{k+\half})} \\
        &\le 2\InNorms{z_{k+1}- w_{k+1}}\InNorms{z_{k+1} - z_{k+\half}}\tag{$F$ is $L$-Lipschitz and $\eta L\le 1$}\\
        & \le 400H \InNorms{z_{k+1} - z_{k+\half}},
    \end{align*}
    where in the last inequality we use $\InNorms{z_{k+1}- w_{k+1}} \le 200H$ is bounded by \Cref{coro:composite-EAG OHM bounded}. 
    
    For $D_k$, since $F+ N_\Z$ is $\rho$-comonotone, we have
    \begin{align*}
        D_k &= 2 \eta \InAngles{z_{k+1} - w_{k+1},  F(w_{k+1}) + d_{k+1} - F(z_{k+1}) - c_{k+1}} \\
        &\le -2\rho\eta \InNorms{F(w_{k+1}) + d_{k+1} - F(z_{k+1}) - c_{k+1}}^2 \\
        &\le -4\rho\eta \InNorms{F(w_{k+1}) + d_{k+1}}^2-4\rho\eta \InNorms{F(z_{k+1}) + c_{k+1}}^2\\
        &\le  -\frac{4\rho}{\eta} \cdot \frac{4\InNorms{z_0 - z^*}^2}{(k+1)^2} -\frac{4\rho}{\eta} \cdot \frac{20 H^2}{(k+1)^2} \tag{convergence rate of \ref{OHM} and \Cref{thm: composite-EAG last-iterate rate}} \\
        &\le \frac{24 H^2}{(k+1)^2}. \tag{$\frac{\rho}{\eta} \ge - \frac{1}{4}$}
    \end{align*} 

    Combining the above bounds for $B_k, C_k, D_k$, we get
    \begin{align*}
        (k+1)^2 A_k \le 17424 H^2 + 400H(k+1)^2\InNorms{z_{k+1} - z_{k+\half}}.
    \end{align*}
    Since $\sum_{t=1}^\infty (t+1)^2 \InNorms{z_{t+1} - z_{t+\half}}^2 \le 3200H^2$, using Cauchy-Schwartz, we get
    \begin{align*}
         \sum_{t=1}^k (t+1)^2 \InNorms{z_{t+1} - z_{t+\half}} &\le \sqrt{ \InParentheses{\sum_{t=1}^k (t+1)^2 \InNorms{z_{t+1} - z_{t+\half}}^2} \cdot \InParentheses{\sum_{t=1}^k (t+1)^2}} \\
         &\le 100H (k+1)^{3/2}.
    \end{align*} 
    Thus $\frac{1}{(k+1)^2} \sum_{t=1}^k (t+1)^2 \InNorms{z_{t+1} - z_{t+\half}} \le 100H (k+1)^{-\frac{1}{2}}$.
    
    Hence we get 
    \begin{align*}
        \frac{1}{(k+1)^2} \sum_{t=1}^k (t+1)^2 A_t = \frac{17424 H^2}{(k+1)^2} + \frac{40000 H^2}{(k+1)^{\half}}
    \end{align*}
    Combining all the above, we conclude that for all $k$, 
    \begin{align*}
        \InNorms{z_{k+1} - w_{k+1}}^2 &\le \frac{\InNorms{z_1 - w_1}^2}{(k+1)^2} + \frac{1}{(k+1)^2}\sum_{t=1}^k (t+1)^2 \InNorms{z_{t+1} - z_{t+\half}}^2 + \frac{1}{(k+1)^2}\sum_{t=1}^k (t+1)^2 \textbf{$A_t$} \\
        &\le \frac{40000 H^2}{(k+1)^{\half}} + \frac{30000H^2}{(k+1)^2}.
    \end{align*}
    This completes the proof.
\end{proof}

\subsubsection{Bounded Iterates of \ref{OHM} and \ref{composite-EAG}}
In this section, we prove auxiliary results needed in the proof of \Cref{thm: composite-EAG merging path}. 
We show that the iterates of \ref{OHM} and \ref{composite-EAG} all have bounded distance away from the initial point $z_0$. 
\begin{lemma}[Bounded Iterates of \ref{OHM}]
\label{lemma:OHM bounded iterates}
    The iterates $\{w_k\}_{k \ge 0}$ of \ref{OHM} satisfies for all $k \ge 0$
    \begin{align*}
        \InNorms{w_{k} - w_0}^2 &\le 4 \InNorms{w_0 - w^*}^2.
    \end{align*}
\end{lemma}
\begin{proof}
    When $k = 1$, by nonexpansiveness of the resolvent $J_{\eta (F + A)}$, we know 
    \[
    \InNorms{w_1 - w^*} = \InNorms{J_{\eta (F + A)}[w_0] - J_{\eta (F + A)}[w^*]} \le \InNorms{w_0 - w^*}. 
    \]
    This implies 
    \[
    \InNorms{w_1 - w_0}^2 \le 2\InNorms{w_1 - w^*}^2 + 2\InNorms{w_0 - w^*}^2 \le 4 \InNorms{w_0 - w^*}^2.
    \]
    For $k \ge 1$, by definitions of \ref{OHM} and $d_{k+1}$, we have
    \begin{align*}
        \InNorms{w_{k+1} - w_0}^2 &= \InNorms{\frac{k}{k+1}(w_k - w_0) - \eta( F(w_{k+1}) + d_{k+1})}^2 \\
        &\le \InParentheses{1 + \frac{1}{k}}\frac{k^2}{(k+1)^2} \InNorms{w_k - w_0}^2 +  (1 + k) \eta^2 \InNorms{( F(w_{k+1}) + d_{k+1})}^2 \tag{Young's inequality} \\
        &\le \frac{k}{k+1} \InNorms{w_k - w_0}^2 + \frac{4\InNorms{w_0 - w^*}^2}{k+1}. \tag{\Cref{thm:OHM}}
    \end{align*}
    This implies $(k+1)\InNorms{w_{k+1} - w_0}^2 \le k \InNorms{w_k - w_0}^2 + 4\InNorms{w_0 - w^*}^2$. By induction, it is easy to see that $\InNorms{w_{k} - w_0}^2 \le 4\InNorms{w_0 - w^*}^2$ for all $k \ge 0$.
\end{proof}

\begin{lemma}[Bounded Iterates of \ref{composite-EAG}]
\label{lemma:composite-EAG bounded iterates}
    The iterates $\{z_k\}_{k \ge 0}$ of \ref{composite-EAG} satisfies for all $k \ge 0$
    \begin{align*}
        \InNorms{z_{k}- z_0}^2 &\le 16040 H^2.
    \end{align*}
\end{lemma}
\begin{proof}
    For $k = 1$, $\InNorms{z_1 - z_0}^2 \le H^2$ is clear.  By definitions of \ref{composite-EAG} and $c_{k+1}$, we have
    \begin{align*}
        \InNorms{z_{k+1} - z_0}^2 &= \InNorms{\frac{k}{k+1}(z_k - z_0) - \eta( F(z_{k+\half}) + d_{k+1})}^2 \\
        &\le \InParentheses{1 + \frac{1}{k}}\frac{k^2}{(k+1)^2} \InNorms{z_k - z_0}^2 +  (1 + k) \eta^2 \InNorms{( F(z_{k+\half}) + d_{k+1})}^2 \tag{Young's inequality} \\
        &\le \frac{k}{k+1} \InNorms{z_k - z_0}^2 + 2(1+k)\eta^2 \InParentheses{ \InNorms{F(z_{k+\half}) - F(z_{k+1})}^2 + \InNorms{F(z_{k+1}) + d_{k+1}}^2 } \\
        &\le \frac{k}{k+1} \InNorms{z_k - z_0}^2 + \frac{1}{k+1}16040 H^2.\tag{\Cref{corollary:composite-EAG bounded summation} and \Cref{thm: composite-EAG last-iterate rate}}
    \end{align*}
    By induction, we get $\InNorms{z_{k+1} - z_0}^2 \le 16040 H^2$.
\end{proof}
Combining \Cref{lemma:OHM bounded iterates} and \Cref{lemma:composite-EAG bounded iterates} directly implies that the iterates of \ref{OHM} and \ref{composite-EAG} are at most constant away from each other.
\begin{corollary}\label{coro:composite-EAG OHM bounded}
    Let $w_0 = z_0$. The iterates $\{z_k\}_{k \ge 0}$ of \ref{composite-EAG} and the iterates $\{w_k\}_{k \ge 0}$ of \ref{OHM} satisfies for all $k \ge 0$,
    \begin{align*}
        \InNorms{w_{k} - z_{k}} \le \sqrt{(32080 + 8) H^2} \le 200H.
    \end{align*}
\end{corollary}

\section{Missing Proofs in \Cref{sec:composite-FEG}}
\label{app:composite FEG}

\subsection{Proof of \Cref{prop:U1 bound}}
\begin{proof}
Note that $z_{\half} = z_0$ and $z_1 = J_{\eta A}\InBrackets{z_0 - \eta F(z_0)}$. Thus we have $\eta c_1 = z_0 - \eta F(z_0) - z_1$. By definition of $U_1$, we have 
\begin{align*}
    U_1 &= \frac{1}{2} \InNorms{\eta F(z_1) + \eta c_1}^2 + \InAngles{\eta F(z_1) + \eta c_1, z_1 - z_0} \\
    & = \frac{1}{2} \InNorms{\eta F(z_1) - \eta F(z_0) + z_0 - z_1}^2 + \InAngles{\eta F(z_1) - \eta F(z_0), z_1 - z_0} - \InNorms{z_1 - z_0}^2 \\
    &= \frac{\eta^2}{2} \InNorms{F(z_1) - F(z_0)}^2 - \frac{1}{2} \InNorms{z_1 - z_0}^2 \\
    &\le \frac{\eta^2 L^2 - 1}{2} \InNorms{z_1 - z_0}^2.
\end{align*}
where we use $\eta c_1 = z_0 - \eta F(z_0) - z_1$ in the second equality, and the $L$-Lipshcitzness of $F$ in fourth equality. If $\eta \le \frac{1}{L}$, then $U_1 \le 0$.
\end{proof}

\subsection{Proof of \Cref{lem:composite-FEG monotone potential}}

\begin{proof}
Fix any $k \ge 1$. We first present several inequalities. Since $F$ is $L$-Lipschitz, we have
\begin{align}\label{eq:new LHS 1}
&\InParentheses{-\frac{(k+1)^2}{2}}\cdot \left( \eta^2 L^2\cdot\InNorms{z_{k+\half}-z_{k+1}}^2 - \InNorms{\eta F(z_{k+\half})-\eta F(z_{k+1})}^2\right)  \le 0. 
\end{align}

Additionally, since $F + A$ is $\rho$-comonotone, $c_k \in A(z_k)$, and $c_{k+1} \in A(z_{k+1})$, we have
\begin{align}\label{eq:new LHS 2}
    \InParentheses{-k(k+1)}\cdot &\biggl ( \InAngles{ \eta F(z_{k+1}) + \eta c_{k+1} - \eta F(z_k) - \eta c_k , z_{k+1} - z_k } \notag  \\
    & \quad\quad\quad- \frac{\rho}{\eta} \InNorms{\eta F(z_{k+1}) + \eta c_{k+1} - \eta F(z_k) - \eta c_k }^2 \biggr )\le 0.
\end{align}

\notshow{
By definition, $ c_k \in A(z_k)$ and $ c_{k+1} \in A(z_{k+1})$. Since $A$ is maximaly monotone, we have
\begin{align}
    (-(k+1)(k+2))\cdot \InAngles{ \eta c_{k+1} - \eta c_k , z_{k+1} - z_k)} \le 0. \label{eq:EAG-G low deg LHS 4}
\end{align}
}

The following identity holds due to Identity \eqref{identity AS} in Proposition~\ref{prop:EAG identity}: we treat $x_0$ as $z_0$; $x_t$ as $z_{k+\frac{t-1}{2}}$ for $t \in \{1,2,3\}$; $y_t$ as $\eta F(z_{k+\frac{t-1}{2}})$ for $t \in \{1,2,3\}$; $u_1$ as $\eta c_k$, and $u_3$ as $\eta c_{k+1}$; $p$ as $\eta^2 L^2$, $q$ as $k$, and $r$ as $\frac{\rho}{\eta}$. Note that by the update rule of \ref{composite-FEG}, we have $\eta c_k =  \frac{z_k +\frac{1}{k+1} (z_0 - z_k) - \frac{k}{k+1}(1 + 2\frac{\rho}{\eta}) \cdot \eta F(z_k) - z_{k+\half}}{\frac{k}{k+1}(1 + 2\frac{\rho}{\eta})}$, and by definition, we have $\eta c_{k+1} = z_k  + \frac{1}{k+1}(z_0 - z_k) -  \eta F(z_{k+\half}) - \frac{2k\frac{\rho}{\eta}}{k+1}\cdot (\eta F(z_k) + \eta c_k)- z_{k+1}$.
\begin{align*}
     & U_k - U_{k+1} + \LHSI~\eqref{eq:new LHS 1} + \LHSI~\eqref{eq:new LHS 2} \notag \\
     =& \frac{(1-\eta^2 L^2)(k+1)^2}{2} \cdot \InNorms{z_{k+1} - z_{k+\half}}^2.
\end{align*}
This completes the proof.
\end{proof}

The following corollary will be used in the proof of \Cref{thm: composite-FEG merging path}.
\begin{corollary}
\label{corollary:composite-FEG bounded summation}
    In the same setup as \Cref{thm: composite-FEG last-iterate rate} but we assume $\rho = 0$, then we have
    \begin{align*}
        \sum_{k=1}^\infty (k+1)^2 \InNorms{z_{k+\half} - z_{k+1}}^2 \le \frac{2 H_0^2}{1 - \eta^2 L^2}.
    \end{align*}
\end{corollary}
\begin{proof}
    By \Cref{lem:composite-FEG monotone potential}, we get $\frac{1-\eta^2 L^2}{2} (k+1)^2 \InNorms{z_{k+\half} - z_{k+1}}^2 \le U_k - U_{k+1}$ holds for any $k \ge 1$. Then the claim follows by telescoping the above inequality and notice that $U_1 - U_{k+1} \le H_0^2$.
\end{proof}

\subsection{Proof of \Cref{lemma:composite-FEG lower bound V_k}}
\begin{proof}
    Fix any $T \ge 1$. According to Lemma~\ref{lem:composite-FEG monotone potential}, we have $U_T \le U_1$. Then by definition of $U_T$, we have
    \begin{align*}
    U_T 
    & = \InParentheses{ \frac{T^2}{2}\InParentheses{1+\frac{2\rho}{\eta}} - \frac{\rho}{\eta} T} \cdot \InNorms{\eta F(z_T) + \eta c_T}^2 + T \cdot \InAngles{\eta F(z_T) + \eta c_T, z_T - z^*} \\& \quad + T \cdot \InAngles{\eta F(z_T) + \eta c_T, z^* - z_0} \\
    & \ge \InParentheses{ \frac{T^2}{2}\InParentheses{1+\frac{2\rho}{\eta}} - \frac{\rho}{\eta} T} \cdot \InNorms{\eta F(z_T) + \eta c_T}^2 + \frac{\rho}{\eta} T \cdot \InNorms{\eta F(z_T) + \eta c_T}^2  \\
    & \quad + T \cdot \InAngles{\eta F(z_T) + \eta c_T, z^* - z_0} \\
    &= \frac{\eta \InParentheses{\eta + 2\rho}T^2}{2} \cdot \InNorms{F(z_T) + c_T}^2 + T \cdot \InAngles{\eta F(z_T) + \eta c_T, z^* - z_0} \\
    & \ge \frac{\eta \InParentheses{\eta + 2\rho}T^2}{2} \cdot \InNorms{F(z_T) + c_T}^2 - \frac{\eta \InParentheses{\eta + 2\rho}T^2}{4} \cdot \InNorms{F(z_T) + c_T}^2 \\
    & \quad - \frac{\eta}{\eta + 2\rho} \InNorms{z_0 - z^*}^2\\ 
    & = \frac{\eta \InParentheses{\eta + 2\rho}T^2}{4} \cdot \InNorms{F(z_T) + c_T}^2 - \frac{\eta}{\eta + 2\rho} \InNorms{z_0 - z^*}^2.
    \end{align*}
In the first inequality, we use the fact that $z^*$ is a solution of the \CMI with the $\rho$-comonotone operator $F+A$.
In the second inequality, we use $\InAngles{a,b} \ge -\frac{\delta}{4} \InNorms{a}^2 - \frac{1}{\delta}\InNorms{b}^2$ for $\delta > 0$.
\end{proof}

\subsection{Proof of \Cref{thm: composite-FEG merging path}}
\begin{proof}
    Using the update rule of \ref{composite-FEG} and \ref{OHM} and recall: $\eta c_{k+1} := z_{k} - \eta F(z_{k+\half}) + \beta_k (z_0 - z_{k}) - z_{k+1}$ and $\eta d_{k+1} = w_{k+\half} - \eta F(w_{k+1}) - w_{k+1}$ we have
    \begin{align*}
        z_{k+1} - w_{k+1} &= \InParentheses{\beta_k z_0 + (1-\beta_k)z_k - \eta F(z_{k+\half}) - \eta c_{k+1}} - \InParentheses{\beta_k w_0 + (1-\beta_k)w_k -\eta F(w_{k+1}) - \eta d_{k+1}}\\
        & = (1-\beta_k) (z_k - w_k) + \eta \InParentheses{F(w_{k+1}) + d_{k+1} - F(z_{k+\half}) - c_{k+1}}, \tag{$w_0 = z_0$}
    \end{align*}
    which implies
    \begin{align*}
        \InNorms{z_{k+1} - w_{k+1}}^2 &= (1-\beta_k)^2 \InNorms{z_k - w_k}^2 +  \eta^2 \InNorms{F(w_{k+1}) + d_{k+1} - F(z_{k+\half}) - c_{k+1}}^2 \\
        & \quad + \underbrace{{2\InAngles{(1-\beta_k)(z_k - w_k), \eta \InParentheses{F(w_{k+1}) + d_{k+1} - F(z_{k+\half}) - c_{k+1}}}}}_{\textbf{I}}.
    \end{align*}
    We focus on term \textbf{I}.  We can verify
    \[
    z_{k+\half} - w_{k+1} = (1-\beta_k)(z_k - w_k) - \eta(1-\beta_k)( F(z_k) + c_{k}) -  \eta (F(w_{k+1}) +d_{k+1}). 
    \]
    Thus term \textbf{I} can be rewritten as 
    \begin{align*}
        \textbf{I} &= 2\InAngles{(1-\beta_k)(z_k - w_k), \eta \InParentheses{F(w_{k+1}) + d_{k+1} - F(z_{k+\half}) - c_{k+1}}} \\
        &=2\InAngles{ z_{k+\half} - w_{k+1} + \eta (1-\beta_k) (F(z_k) + c_{k}) -  \eta(F(w_{k+1})+  d_{k+1}), \eta \InParentheses{F(w_{k+1}) + d_{k+1} - F(z_{k+\half}) - c_{k+1}}} \\
        & = 2\eta^2 \InAngles{(1-\beta_k)(F(z_k) + c_k) -  (F(w_{k+1}) +  d_{k+1}), F(w_{k+1}) + d_{k+1} - F(z_{k+\half}) - c_{k+1}} + \\
        &\quad + \underbrace{2 \eta \InAngles{z_{k+\half} - w_{k+1},  F(w_{k+1}) + d_{k+1} - F(z_{k+\half}) - c_{k+1} } }_{\textbf{$A_k$}}.
    \end{align*}
    Combining the above (we keep term \textbf{$A_k$} for now), we get 
    \begin{align*}
        &\InNorms{z_{k+1} - w_{k+1}}^2\\
        &= (1-\beta_k)^2 \InNorms{z_k - w_k}^2 +  \eta^2 \InNorms{F(w_{k+1}) + d_{k+1} - F(z_{k+\half}) - c_{k+1}}^2 \\
        &\quad + 2\eta^2 \InAngles{ (1-\beta_k)(F(z_k) + c_k) -  (F(w_{k+1}) + d_{k+1}), F(w_{k+1}) + d_{k+1} - F(z_{k+\half}) - c_{k+1}} + \textbf{$A_k$} \\
        &= (1-\beta_k)^2 \InNorms{z_k - w_k}^2 - \eta^2 \InNorms{F(w_{k+1}) + d_{k+1}}^2 + 2\eta^2 (1-\beta_k) \InAngles{F(z_k) + c_k, F(w_{k+1}) + d_{k+1}  } \\
        &\quad - 2\eta^2(1-\beta_k) \InAngles{F(z_k) + c_k, F(z_{k+\half}) + c_{k+1}} + \eta^2 \InNorms{F(z_{k+\half}) + c_{k+1}}^2 + \textbf{$A_k$}. \\
        &\le  (1-\beta_k)^2 \InNorms{z_k - w_k}^2 + \eta^2 \InNorms{(1-\beta_k)(F(z_k) + c_k)}^2 \tag{We use $-a^2 + 2ab \le b^2$} \\
        &\quad - 2\eta^2 (1-\beta_k) \InAngles{F(z_k) + c_k, F(z_{k+\half}) + c_{k+1}} + \eta^2 \InNorms{F(z_{k+\half}) + c_{k+1}}^2 + \textbf{$A_k$} \\
        & = (1-\beta_k)^2 \InNorms{z_k - w_k}^2 + \eta^2 \InNorms{(1-\beta_k)(F(z_k) + c_k) - F(z_{k+\half}) - c_{k+1}}^2 + \textbf{$A_k$} \\
        & = (1-\beta_k)^2 \InNorms{z_k - w_k}^2 +  \InNorms{z_{k+1} - z_{k+\half}}^2 + \textbf{$A_k$},
    \end{align*}
    where in the last equality we use the fact that $z_{k+1} - z_{k+\half} = \eta(1-\beta_k)(F(z_k) + c_k)  - \eta F(z_{k+\half}) - \eta c_{k+1}$ by update rule of \ref{composite-FEG}. Plugging $\beta_k = \frac{1}{k+1}$ and multiplying both sides with $(k+1)^2$ gives
    \[
    (k+1)^2\InNorms{z_{k+1} - w_{k+1}}^2 \le k^2\InNorms{z_k - w_k}^2 + (k+1)^2 \InNorms{z_{k+1} - z_{k+\half}}^2 + (k+1)^2\textbf{$A_k$}.
    \]
    Telescoping the above gives 
    \begin{align*}
        &(k+1)^2\InNorms{z_{k+1} - w_{k+1}}^2 \le \InNorms{z_1 - w_1}^2 + \sum_{t=1}^k (t+1)^2 \InNorms{z_{t+1} - z_{t+\half}}^2 + \sum_{t=1}^k (t+1)^2 \textbf{$A_t$}. \\ 
        &\Rightarrow \InNorms{z_{k+1} - w_{k+1}}^2 \le \frac{\InNorms{z_1 - w_1}^2}{(k+1)^2} + \frac{1}{(k+1)^2}\sum_{t=1}^k (t+1)^2 \InNorms{z_{t+1} - z_{t+\half}}^2 + \frac{1}{(k+1)^2}\sum_{t=1}^k (t+1)^2 \textbf{$A_t$}
    \end{align*}
    It remains to bound $\frac{1}{(k+1)^2}\sum_{t=1}^k (t+1)^2 \InNorms{z_{t+1} - z_{t+\half}}^2$ and $\frac{1}{(k+1)^2}\sum_{t=1}^k (t+1)^2 A_t$. By \Cref{corollary:composite-FEG bounded summation}, we have $\sum_{t=1}^\infty(t+1)^2 \InNorms{z_{t+1} - z_{t+\half}}^2 \le \frac{2H_0^2}{1 - \eta^2 L^2}$.

    For the term with $A_t$, we need a more careful analysis.  We decompose $A_k = B_k + C_k + D_k$ as follows.
    \begin{align*}
        A_k &= 2 \eta \InAngles{z_{k+\half} - w_{k+1},  F(w_{k+1}) + d_{k+1} - F(z_{k+\half}) - c_{k+1}} \\
        &= 2 \eta \InAngles{z_{k+\half} -  z_{k+1},  F(w_{k+1}) + d_{k+1} - F(z_{k+\half}) - c_{k+1}} \\
        & \quad + 2\eta \InAngles{z_{k+1} - w_{k+1}, F(w_{k+1}) + d_{k+1} - F(z_{k+\half}) - c_{k+1}} \\
        &= \underbrace{2 \eta \InAngles{z_{k+\half} - z_{k+1},  F(w_{k+1}) + d_{k+1} - F(z_{k+\half}) - c_{k+1}}}_{B_k} + \underbrace{2\eta \InAngles{z_{k+1} - w_{k+1}, F(z_{k+1}) - F(z_{k+\half})}}_{C_k} \\
        &\quad  + \underbrace{2\eta \InAngles{ z_{k+1} - w_{k+1}, F(w_{k+1}) + d_{k+1} - F(z_{k+1}) - c_{k+1} }}_{D_k}.
    \end{align*}

    For $B_k$, we have
    \begin{align*}
        B_k &= 2\eta \InAngles{ z_{k+\half} - z_{k+1}, F(w_{k+1}) + d_{k+1} - F(z_{k+\half}) - c_{k+1} } \\
        &\le 2 \eta \InNorms{z_{k+1} - z_{k+\half}} \InParentheses{\InNorms{F(w_{k+1}) + d_{k+1}} + \InNorms{F(z_{k+1}) - F(z_{k+\half})}+\InNorms{F(z_{k+1}) + c_{k+1}}}\\
        & \le 2\eta L \InNorms{z_{k+1} - z_{k+\half}}^2 + 2 \eta \InNorms{z_{k+1} - z_{k+\half}} \InParentheses{\InNorms{F(w_{k+1}) + d_{k+1}} +\InNorms{F(z_{k+1}) + c_{k+1}}}\\
        &\le  \frac{4 H_0^2}{(1-\eta^2 L^2)(k+1)^2} + 2\eta \sqrt{\frac{2H_0^2}{(1-\eta^2 L^2)(k+1)^2}} \cdot \InParentheses{ \frac{2\InNorms{z_0 - z^*}}{\eta (k+1)} + \frac{2 H_0}{\eta(k+1)} } \\
        &\le \frac{20 H_0^2}{(1-\eta^2 L^2)(k+1)^2},
    \end{align*}
    where in the second last inequality we use $\eta L \le 1$, the last-iterate convergence rate of \ref{OHM} (\Cref{thm:OHM}) and \ref{composite-FEG} (\Cref{thm: composite-FEG last-iterate rate}) as well as the bound on $\InNorms{z_{k+1} - z_{k+\half}}^2$ (\Cref{corollary:composite-FEG bounded summation}); in the last inequality, we use $ H_0 \ge \InNorms{z_0 -z^*}$.

    For $C_k$, we have 
    \begin{align*}
        C_k &= 2\eta \InAngles{z_{k+1} - w_{k+1}, F(z_{k+1}) - F(z_{k+\half})} \\
        &\le 2\eta \InNorms{z_{k+1} - w_{k+1}} \InNorms{ F(z_{k+1}) -  F(z_{k+\half})} \\
        &\le 2\InNorms{z_{k+1}- w_{k+1}}\InNorms{z_{k+1} - z_{k+\half}}\tag{$F$ is $L$-Lipschitz and $\eta L\le 1$}\\
        & \le \frac{12H_0}{\sqrt{1-\eta^2 L^2}} \InNorms{z_{k+1} - z_{k+\half}},
    \end{align*}
    where in the last inequality we use $\InNorms{z_{k+1}- w_{k+1}} \le \frac{6H_0}{\sqrt{1-\eta^2 L^2}}$ is bounded by \Cref{coro:composite-FEG OHM bounded}. 
    
    For $D_k$, since $F+ N_\Z$ is $\rho$-comonotone, we have
    \begin{align*}
        D_k &= 2 \eta \InAngles{z_{k+1} - w_{k+1},  F(w_{k+1}) + d_{k+1} - F(z_{k+1}) - c_{k+1}} \\
        &\le -2\rho\eta \InNorms{F(w_{k+1}) + d_{k+1} - F(z_{k+1}) - c_{k+1}}^2 \\
        &\le -4\rho\eta \InNorms{F(w_{k+1}) + d_{k+1}}^2-4\rho\eta \InNorms{F(z_{k+1}) + c_{k+1}}^2\\
        &\le  -\frac{4\rho}{\eta} \cdot \frac{4\InNorms{z_0 - z^*}^2}{(k+1)^2} -\frac{4\rho}{\eta} \cdot \frac{4 H_0^2}{(1-\eta^2 L^2)(k+1)^2} \tag{convergence rate of \ref{OHM} and \Cref{thm: composite-FEG last-iterate rate}} \\
        &\le \frac{8 H_0^2}{(1-\eta^2 L^2)(k+1)^2}. \tag{$\frac{\rho}{\eta} \ge - \frac{1}{4}$}
    \end{align*} 

    Combining the above bounds for $B_k, C_k, D_k$, we get
    \begin{align*}
        (k+1)^2 A_k \le \frac{28 H_0^2}{1-\eta^2 L^2} + \frac{12H_0}{\sqrt{1-\eta^2 L^2}} (k+1)^2\InNorms{z_{k+1} - z_{k+\half}}.
    \end{align*}
    Since $\sum_{t=1}^\infty (t+1)^2 \InNorms{z_{t+1} - z_{t+\half}}^2 \le \frac{2H_0^2}{1-\eta^2 L^2}$, using Cauchy-Schwartz, we get
    \begin{align*}
         \sum_{t=1}^k (t+1)^2 \InNorms{z_{t+1} - z_{t+\half}} &\le \sqrt{ \InParentheses{\sum_{t=1}^k (t+1)^2 \InNorms{z_{t+1} - z_{t+\half}}^2} \cdot \InParentheses{\sum_{t=1}^k (t+1)^2}} \\
         &\le \frac{2 H_0}{\sqrt{1-\eta^2 L^2}} (k+1)^{3/2}.
    \end{align*} 
    Thus $\frac{12H_0}{\sqrt{1-\eta^2 L^2}}\sum_{t=1}^k (t+1)^2 \InNorms{z_{t+1} - z_{t+\half}} \le \frac{24 H_0^2}{(1-\eta^2 L^2)}(k+1)^{3/2}$.
    
    Hence we get 
    \begin{align*}
        \frac{1}{(k+1)^2} \sum_{t=1}^k (t+1)^2 A_t = \frac{28 H_0^2}{(1-\eta^2 L^2)(k+1)^2} + \frac{24 H_0^2}{(1-\eta^2 L^2)(k+1)^{\half}}.
    \end{align*}
    Combining all the above, we concludes that for all $k$, 
    \begin{align*}
        \InNorms{z_{k+1} - w_{k+1}}^2 &\le \frac{\InNorms{z_1 - w_1}^2}{(k+1)^2} + \frac{1}{(k+1)^2}\sum_{t=1}^k (t+1)^2 \InNorms{z_{t+1} - z_{t+\half}}^2 + \frac{1}{(k+1)^2}\sum_{t=1}^k (t+1)^2 \textbf{$A_t$} \\
        &\le \frac{H_0^2}{1-\eta^2 L^2} \cdot\InParentheses{\frac{24}{(k+1)^{\half}} + \frac{36}{(k+1)^2}}.
    \end{align*}
    This completes the proof.
\end{proof}

\subsubsection{Bounded Iterates of \ref{composite-FEG}}
In this section, we prove auxiliary results needed in the proof of \Cref{thm: composite-FEG merging path}. 
We show that the iterates of \ref{OHM} and \ref{composite-FEG} all have bounded distance away from the initial point $z_0$. 
\begin{lemma}[Bounded Iterates of \ref{composite-FEG}]
\label{lemma:composite-FEG bounded iterates}
    Let $\rho = 0$. The iterates $\{z_k\}_{k \ge 0}$ of \ref{composite-FEG} satisfies for all $k \ge 0$
    \begin{align*}
        \InNorms{z_{k}- z_0}^2 &\le \frac{12H_0^2}{1-\eta^2 L^2}.
    \end{align*}
\end{lemma}
\begin{proof}
    For $k = 1$, $\InNorms{z_1 - z_0}^2 \le H_0^2$ is clear.  By definitions of \ref{composite-FEG} and $c_{k+1}$, we have
    \begin{align*}
        \InNorms{z_{k+1} - z_0}^2 &= \InNorms{\frac{k}{k+1}(z_k - z_0) - \eta( F(z_{k+\half}) + d_{k+1})}^2 \\
        &\le \InParentheses{1 + \frac{1}{k}}\frac{k^2}{(k+1)^2} \InNorms{z_k - z_0}^2 +  (1 + k) \eta^2 \InNorms{( F(z_{k+\half}) + d_{k+1})}^2 \tag{Young's inequality} \\
        &\le \frac{k}{k+1} \InNorms{z_k - z_0}^2 + 2(1+k)\eta^2 \InParentheses{ \InNorms{F(z_{k+\half}) - F(z_{k+1})}^2 + \InNorms{F(z_{k+1}) + d_{k+1}}^2 } \\
        &\le \frac{k}{k+1} \InNorms{z_k - z_0}^2 + \frac{1}{k+1} \cdot \frac{12H_0^2}{1-\eta^2 L^2}.\tag{\Cref{corollary:composite-FEG bounded summation} and \Cref{thm: composite-FEG last-iterate rate}}
    \end{align*}
    By induction, we get $\InNorms{z_{k+1} - z_0}^2 \le \frac{12H_0^2}{1-\eta^2 L^2}$.
\end{proof}
Combining \Cref{lemma:composite-FEG bounded iterates} and \Cref{lemma:OHM bounded iterates} gives the following cororllary.
\begin{corollary}\label{coro:composite-FEG OHM bounded}
    Let $w_0 = z_0$. The iterates $\{z_k\}_{k \ge 0}$ of \ref{composite-FEG} and the iterates $\{w_k\}_{k \ge 0}$ of \ref{OHM} satisfies for all $k \ge 0$,
    \begin{align*}
        \InNorms{w_{k} - z_{k}} \le \frac{6H_0}{\sqrt{1-\eta^2 L^2}}.
    \end{align*}
\end{corollary}

\section{proj-EAG for comonotone inclusion with point convergence}\label{sec:proj-EAG}
In this section, we study constrained variational inequality problem, a special case of \CMI where $A = \partial \indSet_\Z = N_\Z$ is the normal cone operator of a closed convex set $\Z \subseteq\R^n$. This problem captures constrained non-convex-non-concave min-max optimization. We extend the extra anchored gradient (EAG) algorithm in the unconstrained case to this setting and propose the following projected extra anchored gradient (\ref{proj-EAG}) algorithm.
\begin{equation}\label{proj-EAG}\tag{proj-EAG}
    \begin{aligned}
        z_{k+\half} &= \Pi_\Z[\beta_k z_0 + (1- \beta_k) z_k - \eta F(z_k)] \\
        z_{k+1} &= \Pi_\Z[\beta_k z_0 + (1- \beta_k) z_k - \eta F(z_{k+\half})]
    \end{aligned} 
\end{equation}
where $\beta_k = \frac{1}{k+1}$. Note that all the iterates $\{z_k, z_{k+\half}\}_{k \ge 0}$ lies in the feasible set $\Z \subseteq \R^n$. As a consequence, we only require $L$-Lipschitzness and $\rho$-comonotonicity hold on the feasible set $\Z$, but not the whole $\R^n$ space. Specifically, we assume $F: \Z \rightarrow \Z$ is $L$-Lipschitz and $F + N_\Z$ is $\rho$-comonotone  with $\rho \ge -\frac{1}{20L}$.

The two main results of the section are:
\begin{itemize}
    \item[1.] \ref{proj-EAG} has $O(\frac{1}{T})$ last-iterate convergence rate with respect to the tangent residual. 
    \item[2.] \ref{proj-EAG} has point convergence
\end{itemize}

\subsection{Last-Iterate Convergence Rate}
We use the following potential function for \[V_k := \frac{k(k+1)}{2}\cdot \InNorms{\eta F(z_k) + \eta c_k}^2 + k\cdot \InAngles{\eta F(z_k) + \eta c_k, z_k - z_0},\qquad k\ge 1,\]
where $c_k := \frac{\beta_{k-1} z_0 + (1-\beta_{k-1})z_{k-1} - \eta F(z_{k-\half}) - z_k}{\eta}$ for $k \ge 1$. Introducing $c_{k}$ gives the identity $z_{k+1} = \beta_k z_0 + (1- \beta_k) z_k - \eta F(z_{k+\half}) - \eta c_{k+1}$ for all $k \ge 0$. Recall that by definition of projection on $\Z$, we have $c_k \in N_\Z(z_k)$. 

\begin{lemma}[Upper bound of $V_1$]\label{lemma:proj-EAG V_1 bound}
    In the same setup as \Cref{thm:proj-EAG monotone potential at constrained}, we have $V_1 \le  16\InNorms{z_0 - z_\half}^2 \le 16 \eta^2 r^{tan}_{F, N_\Z}(z_0)^2$ 
\end{lemma}
\begin{proof}
    By nonexpansiveness of $\Pi_\Z$, $L$-Lipschitzness of $F$, and $\eta L \le 1$, we have $\InNorms{z_1 - z_\half}^2 \le \InNorms{z_\half - z_0}^2$. This also implies $\InNorms{z_1 - z_0}^2 \le 2\InNorms{z_1 - z_\half}^2 + 2 \InNorms{z_\half - z_0}^2 \le 4 \InNorms{z_\half - z_0}^2$. Then we get 
    \begin{align*}
        \InNorms{\eta F(z_1) + \eta c_1}^2 & = \InNorms{z_0 - z_1 + \eta F(z_1) - \eta F(z_\half)}^2 \\
        &\le 2\InNorms{z_0 - z_1}^2 + 2 \InNorms{z_\half - z_1}^2 \\
        &\le 10\InNorms{z_0 - z_\half}^2.
    \end{align*}
    Now by definition of $V_1$,
    \begin{align*}
        V_1 &= \InNorms{\eta (F(z_1) + c_1)}^2 + \eta \InAngles{F(z_1) + c_1, z_1 - z_0} \\
        &\le (10 +2\sqrt{10})\InNorms{z_0 - z_\half}^2 \le 16\InNorms{z_0 - z_\half}^2.
    \end{align*}
    Let $c_0 \in N_\Z(z_0)$. Using nonexpansiveness of $J_{\eta A}$, we have $\InNorms{z_\half - z_0} \le \InNorms{\eta F(z_0) + \eta c_0}$. Thus $\InNorms{z_\half - z_0} \le \eta r^{tan}_{F, N_\Z}(z_0)$.
\end{proof}

\begin{theorem}[Approximate monotonicity of the potential]\label{thm:proj-EAG monotone potential at constrained}
Let $\Z \subseteq \R^n$ is closed convex set and suppose $F:\Z \rightarrow \R^n$ is $L$-Lispchitz and $F + N_\Z$ is a $\rho$-comonotone operator with $0\ge \rho \ge  -\frac{1}{20L}$. Let $z_0 \in \Z$ be any starting point and $\{z_k, z_{k+\half}\}_{k\ge 0}$ be the iterates of \ref{proj-EAG} with step size $\eta = \frac{0.31}{L}$ that satisfies \Cref{fact:step size}. Then for any $k \ge 1$, \[
V_{k+1} \le V_k + \frac{1}{2} \InNorms{\eta F(z_{k+1}) + \eta c_{k+1}}^2 - \frac{9k(k+1)}{4000}\InNorms{z_{k+\half} - z_{k+1}}^2.
\]
\end{theorem}
\begin{proof}
We first present several inequalities. Let $p = \frac{1}{3}$ and $c = -\frac{4p\rho}{\eta} \ge 0$. Since $F$ is $L$-Lipschitz,  we have
\begin{align*}
    \eta^2 L^2 \cdot \InNorms{z_{k+\half} - z_{k+1}}^2 - \InNorms{\eta F(z_{k+\half}) - \eta F(z_{k+1})}^2 \ge 0.
\end{align*}
Multiplying the both sides of the above inequality by $(1+c)$ and rearranging terms, we get
\begin{align}
&p \cdot \InNorms{z_{k+\half} - z_{k+1}}^2 - \InNorms{\eta F(z_{k+\half}) - \eta F(z_{k+1})}^2 \nonumber \\
& + \InParentheses{(1+c)\eta^2 L^2 - p} \cdot \InNorms{z_{k+\half}}^2 - c \cdot \InNorms{\eta F(z_{k+\half}) - \eta F(z_{k+1})}^2 \ge 0. \label{eq:proj-EAG low deg LHS 2}.
\end{align}
Since $F + N_\Z$ satisfies negative comonotonicity and $0 \in N_\Z(z)$ for any $z \in \Z$, we have
\begin{align}
   & (-k(k+1))\cdot \InParentheses{ \InAngles{ \eta F(z_{k+1}) - \eta F(z_k) , z_{k+1} - z_k)} - \frac{\rho}{\eta} \cdot\InNorms{\eta F(z_{k+1}) - \eta F(z_k)}^2 }\le 0. \label{eq:proj-EAG low deg LHS 3}
\end{align}

Since $z_{k+\half}=\Pi_{\Z}[z_k-\eta F(z_k) + \frac{1}{k+1}(z_0 - z_k)]$,
we can infer that $z_k-\eta F(z_k) +\frac{1}{k+1}(z_0 - z_k) - z_{k+\half}\in N_\Z(z_{k+\half})$. Moreover, by definition of $c_k$ and $c_{k+1}$, we know $c_k \in N_\Z(z_k)$ and $c_{k+1} \in N_\Z(z_{k+1})$. Therefore, we have
\begin{align}
   & (-k(k+1))\cdot \InAngles{ z_k - \eta F(z_k) - z_{k+\half} +\frac{1}{k+1}(z_0 - z_k),z_{k+\half} - z_{k+1} } \le 0, \label{eq:proj-EAG low deg LHS 4}\\
   & (-k(k+1))\cdot \InAngles{ \eta c_{k+1} ,z_{k+1} - z_k } \le 0, \label{eq:proj-EAG low deg LHS 5}\\
   & (-k(k+1))\cdot \InAngles{ \eta c_k,z_{k} - z_{k+\half} } \le 0. \label{eq:proj-EAG low deg LHS 6}
\end{align}

Moreover, by definition, we have $\eta c_{k+1} = z_k - \eta F(z_{k+\half}) + \frac{1}{k+1}(z_0 - z_k) - z_{k+1}$). Then have the following identity holds. The correctness of the identity follows by \Cref{eq:identity proj-EAG} in \Cref{prop:EAG identity}:  we treat $x_0$ as $z_0$; $x_t$ as $z_{k+\frac{t-1}{2}}$ for $t \in \{1,2,3\}$; $y_t$ as $\eta F(z_{k+\frac{t-1}{2}}$ for $t \in \{1,2,3\}$; $u_1$ as $\eta c_k$ and $u_3$ as $\eta c_{k+1}$; $q$ as $k$. Note that Term \eqref{eq:proj-EAG low deg RHS 1}, \eqref{eq:proj-EAG low deg RHS 2}, and \eqref{eq:proj-EAG low deg RHS 3} comes from the right hand side of \Cref{eq:identity proj-EAG}, while Term \eqref{eq:proj-EAG low deg RHS 4} directly comes from Inequality \eqref{eq:proj-EAG low deg LHS 2} and Term \eqref{eq:proj-EAG low deg RHS 5} directly comes from Inequality \eqref{eq:proj-EAG low deg LHS 3}.
\begin{align}
     & V_k - V_{k+1} - \frac{k(k+1)}{2p} \cdot \LHSI~\eqref{eq:proj-EAG low deg LHS 2} + \LHSI~\eqref{eq:proj-EAG low deg LHS 3} + \LHSI~\eqref{eq:proj-EAG low deg LHS 4}  \notag \\
     \quad & +\LHSI~\eqref{eq:proj-EAG low deg LHS 5} + \LHSI~\eqref{eq:proj-EAG low deg LHS 6}   \notag\\
     =& \frac{k(k+1)}{2}\cdot \InNorms{z_{k+\half} - z_k + \eta F(z_k) + \eta c_k + \frac{1}{k+1}(z_k - z_0)}^2 \label{eq:proj-EAG low deg RHS 1} \\
     \quad &+ \frac{(1-p)k(k+1)}{2p}\cdot \InNorms{\eta F(z_{k+\half}) - \eta F(z_{k+1})}^2 \label{eq:proj-EAG low deg RHS 2}\\
     \quad & + (k+1) \cdot \InAngles{\eta F(z_{k+\half}) - \eta F(z_{k+1}), \eta F(z_{k+1}) + \eta c_{k+1}}\label{eq:proj-EAG low deg RHS 3} \\
     \quad &+ \frac{k(k+1)}{2p} \cdot \InParentheses{(p - (1+c)\eta^2 L^2) \InNorms{z_{k+\half} - z_{k+1}}^2 + c\cdot \InNorms{\eta F(z_{k+\half})-\eta F(z_{k+1})}^2} \label{eq:proj-EAG low deg RHS 4} \\
     \quad &+  k(k+1)\frac{\rho}{\eta}\InNorms{\eta F(z_{k+1}) - \eta F(z_k)}^2.\label{eq:proj-EAG low deg RHS 5} 
\end{align}
Since $\InNorms{a}^2 + \InAngles{a,b} = \InNorms{a+\frac{b}{2}}^2 -\frac{\InNorms{b}^2}{4}$, we have
\begin{align*}
    &\text{Expression} (\ref{eq:proj-EAG low deg RHS 2}) + \text{Expression} (\ref{eq:proj-EAG low deg RHS 3}) \\
    =& \InNorms{\sqrt{\frac{(1-p)k(k+1)}{2p}}\cdot \InParentheses{\eta F(z_{k+\half}) - \eta F(z_{k+1})} + \sqrt{\frac{p(k+1)}{2(1-p) k}} \cdot \InParentheses{\eta F(z_{k+1}) + \eta c_{k+1}}}^2 \notag \\
     & - \frac{k+1}{2k} \cdot \frac{p}{1-p} \InNorms{\eta F(z_{k+1}) + \eta c_{k+1}}^2\\ 
     & \ge -\frac{p}{1-p} \InNorms{\eta F(z_{k+1}) + \eta c_{k+1}}^2. \tag{$k \ge 1$}\\
     & = - \frac{1}{2} \InNorms{\eta F(z_{k+1}) + \eta c_{k+1}}^2.
\end{align*}

Now it remains to give a non-negative lower bound of Expression \eqref{eq:proj-EAG low deg RHS 4} + Expression \eqref{eq:proj-EAG low deg RHS 5}. Recall that $p = \frac{1}{3}$ and $c = -\frac{4\rho p}{\eta}$, thus 
\begin{align*}
   &(\frac{2}{k(k+1)}) \cdot \InParentheses{ \text{Expression } \eqref{eq:proj-EAG low deg RHS 4} + \text{Expression } \eqref{eq:proj-EAG low deg RHS 5}} \\
   &= \InParentheses{1 - \InParentheses{3 - \frac{4\rho}{\eta}}\cdot \eta^2 L^2}\cdot \InNorms{z_{k+\half} - z_{k+1}}^2 - \frac{4\rho}{\eta} \cdot \InParentheses{ \cdot \InNorms{\eta F(z_{k+\half})-\eta F(z_{k+1})}^2 - \frac{1}{2}\cdot \InNorms{\eta F(z_{k+1}) - \eta F(z_k)}^2  } \\
   & \ge \InParentheses{1 - \InParentheses{3 - \frac{4\rho}{\eta}}\cdot \eta^2 L^2}\cdot \InNorms{z_{k+\half} - z_{k+1}}^2 + \frac{4\rho}{\eta}\cdot \InNorms{\eta F(z_{k+\half}) - \eta F(z_k)}^2 
   \tag{$  \InNorms{A}^2 - \frac{1}{2}\InNorms{B}^2 \ge -\InNorms{A-B}^2$}\\
   & \ge \InParentheses{1 + \frac{4\rho}{\eta} - \InParentheses{3 - \frac{4\rho}{\eta}}\cdot \eta^2 L^2} \InNorms{z_{k+\half} - z_{k+1}}^2,
\end{align*}
where the last inequality holds because $\InNorms{z_{k+\half} - z_{k+1}}^2 \le \InNorms{\eta F(z_{k+\half}) - \eta F(z_k)}^2$, which is due to update rule of \ref{proj-EAG} and the fact that  $\Pi_Z$ is non-expansive. 
Hence, we have $V_{k+1} \le V_k + \frac{1}{2} \InNorms{\eta F(z_{k+1}) + \eta c_{k+1}}^2$ -  $\InParentheses{1 + \frac{4\rho}{\eta} - \InParentheses{3 - \frac{4\rho}{\eta}}\cdot \eta^2 L^2} \frac{k(k+1)}{2} \InNorms{z_{k+\half} - z_{k+1}}^2$. Plugging $\eta = \frac{0.31}{L}$, we have  $\InParentheses{1 + \frac{4\rho}{\eta} - \InParentheses{3 - \frac{4\rho}{\eta}}\cdot \eta^2 L^2} \ge \frac{9}{2000}$ ( by \Cref{fact:step size}). This completes the proof.
\end{proof}

We also show that $V_k$ is of order $\Omega(k^2 \cdot r^{tan}(z_k))$. 
\begin{lemma}\label{lemma:proj-EAG lower bound V_k}
    In the same setup as \Cref{thm:proj-EAG monotone potential at constrained}, for any $k \ge 1$, we have
    \[
    \frac{k(k+1)}{4} \InNorms{\eta F(z_k) + \eta c_k}^2 \le V_k + \InNorms{z_0 - z^*}^2
    \]
    In particular, $V_k \ge - \InNorms{z_0 -z^*}^2$.
\end{lemma}
\begin{proof}
    Fix any $k \ge 1$. Since $0 \in F(z^*) + N_\Z(z^*)$, by $\rho$-comonotonicity and \Cref{fact:step size}, we have 
    \begin{equation}
        \InAngles{\eta F(z_k) + \eta c_k, z_k - z^*} \ge \frac{\rho}{\eta} \InNorms{\eta F(z_k) + \eta c_k}^2 > -\frac{1}{4} \InNorms{\eta F(z_k) + \eta c_k}^2. \label{eq:comonotone-1}
    \end{equation}
    By definition of $V_k$, we have
    \begin{align*}
        V_k &= \frac{k(k+1)}{2}\cdot \InNorms{\eta F(z_k) + \eta c_k}^2 + k\cdot \InAngles{\eta F(z_k) + \eta c_k, z_k - z_0} \\
        &=  \frac{k(k+1)}{2}\cdot \InNorms{\eta F(z_k) + \eta c_k}^2 + k\cdot \InAngles{\eta F(z_k) + \eta c_k, z_k - z^*} + k\InAngles{\eta F(z_k) + \eta c_k, z^* - z_0} \\
        &\ge \frac{k(k+1)}{2}\cdot \InNorms{\eta F(z_k) + \eta c_k}^2 -\frac{1}{4}\InNorms{\eta F(z_k) + \eta c_k}^2 + k\InAngles{\eta F(z_k) + \eta c_k, z^* - z_0}  \tag{By \Cref{eq:comonotone-1}} \\
        &\ge \frac{k(k+\frac{1}{2})}{2}\cdot \InNorms{\eta F(z_k) + \eta c_k}^2 - \frac{k^2}{4} \InNorms{\eta F(z_k) + \eta c_k}^2 - \InNorms{z_0 - z^*}^2 \\
        &= \frac{k(k+1)}{4}\cdot \InNorms{\eta F(z_k) + \eta c_k}^2 - \InNorms{z_0 - z^*}^2,
    \end{align*}
    where in the second last inequality, we apply $\InAngles{a, b} \ge -\frac{\alpha}{4}\InNorms{a^2} - \frac{1}{\alpha}\InNorms{b^2}$ with $a = \eta F(z_k) + \eta c_k$, $b = z^* -z_0$, and $\alpha = k$. This completes the proof.
\end{proof}

Combinging the approximate monotonicity of the potential function $V_k$ in \Cref{thm:proj-EAG monotone potential at constrained} and the above lower bound of $V_k$, we can prove the $O(\frac{1}{T})$ last-iterate convergence rate easily. 
\begin{theorem}[Last-iterate convergence rate of \eqref{proj-EAG} for comonotone inclusion problem]\label{thm:proj-EAG last rate}
    In the same setup as \Cref{thm:proj-EAG monotone potential at constrained}, we have for any $k \ge 1$,
    \[
    \InNorms{F(z_k) + c_k}^2 \le \frac{20(V_1 + \InNorms{z_0 -z^*}^2)}{\eta ^2 k^2} \le  \frac{20H^2}{\eta ^2 k^2}.
    \]
    where $H^2 = 16\InNorms{z_\half - z_0}^2 + \InNorms{z_0 - z^*}^2 \le 16\eta^2 r^{tan}_{F, N_\Z}(z_0)^2 + \InNorms{z_0 - z^*}^2$. 
\end{theorem}
\begin{proof}
    For $k = 1$, from \Cref{lemma:proj-EAG lower bound V_k}, we directly get that $\InNorms{\eta F(z_1) + \eta c_1}^2 \le 4(V_1 + \InNorms{z_0 - z^*}^2)$.
    
    Now fix any $k \ge 2$. Combining \Cref{thm:proj-EAG monotone potential at constrained} and \Cref{lemma:proj-EAG lower bound V_k}, we have 
    \begin{align*}
        \frac{k(k+1)}{4}\InNorms{\eta F(z_k) + \eta c_k}^2 &\le V_k + \InNorms{z_0 - z^*}^2\\
        &\le V_1 +  \InNorms{z_0 - z^*}^2 + \frac{1}{2} \cdot \sum_{t=2}^k \InNorms{\eta F(z_t) + \eta c_t}^2.
    \end{align*}
    Subtracting $\frac{1}{2}\InNorms{F(z_k) + \eta c_k}^2$ from both sides and noting that $\frac{1}{4}k \ge \frac{1}{2}$ for $k \ge 2$ gives
    \[
    \frac{k^2}{4} \InNorms{\eta F(z_k) + \eta c_k}^2 \le V_1 +  \InNorms{z_0 - z^*}^2 + \frac{1}{2} \cdot \sum_{t=2}^{k-1} \InNorms{\eta F(z_t) + \eta c_t}^2
    \]
    Now we can apply \Cref{prop:sequence analysis} with $C_1 = V_1 + \InNorms{z_0 - z^*}^2$  to conclude that $\InNorms{\eta F(z_k) + \eta c_k}^2 \le \frac{20(V_1 + \InNorms{z_0 - z^*}^2)}{k^2}$. Applying \Cref{lemma:proj-EAG V_1 bound}, we know $V_1 \le 16\InNorms{z_\half - z_0}^2  \le 16\eta^2 r^{tan}_{F, N_\Z}(z_0)^2$ gives the desired result. 
\end{proof}

\begin{corollary}\label{corollary:proj-EAG bounded summation}
    In the same setup as \Cref{thm:proj-EAG monotone potential at constrained}, we have
    \[
     \sum_{k=1}^\infty k(k+1) \InNorms{z_{k+1} - z_{k+\half}}^2 \le 4000 H^2.
    \]
    It also implies $\InNorms{z_{k+1} - z_{k+\half}}^2 \le \frac{8000 H^2}{(k+1)^2}$ for all $k \ge 1$.
\end{corollary}
\begin{proof}
    By \Cref{thm:proj-EAG monotone potential at constrained}, we have
    \[
    \frac{9k(k+1)}{4000}\InNorms{z_{k+1} - z_{k+\half}}^2 \le V_k - V_{k+1} + \frac{1}{2}\InNorms{\eta F(z_{k+1})+\eta c_{k+1}}^2
    \]
    Telescoping the above inequality for $k = 1, 2, \ldots, T$ gives
    \begin{align*}
        \frac{9}{4000} \sum_{k=1}^T k(k+1) \InNorms{z_{k+1} - z_{k+\half}}^2 &\le V_1- V_{T+1} + \frac{1}{2} \sum_{k=1}^T \InNorms{\eta F(z_{k+1})+\eta c_{k+1}}^2 \\
        &\le V_1 + \InNorms{z_0 -z^*}^2 + 10H^2 \cdot \InParentheses{\sum_{k=1}^T \frac{1}{(t+1)^2}} \\
        & \le 9H^2.
    \end{align*}
    where in the second inequality we use (1) $V_{T+1} \ge -\InNorms{z_0 - z^*}^2$ by \Cref{lemma:proj-EAG lower bound V_k} and (2) $\InNorms{\eta F(z_{k+1})+\eta c_{k+1}}^2 \le \frac{20H^2}{(t+1)^2}$ by \Cref{thm:proj-EAG monotone potential at constrained}, in the last inequality we use $\sum_{k=1}^\infty \frac{1}{(t+1)^2} = \frac{\pi^2}{6} -1 \le \frac{3}{4}$. This concludes $\sum_{k=1}^\infty k(k+1) \InNorms{z_{k+1} - z_{k+\half}}^2 \le 4000 H^2$. Since $\frac{1}{2}(k+1)^2 \le k(k+1)$ for all $k \ge 1$, it further implies $\InNorms{z_{k+1} - z_{k+\half}}^2 \le \frac{8000 H^2}{(k+1)^2}$ for all $k \ge 1$.
\end{proof}

\subsection{Point Convergence of \ref{proj-EAG} in the Comonotone Case}

In the following, we show that $\InNorms{z_{k+1} - w_{k+1}} \rightarrow 0$, i.e., the trajectory of \ref{proj-EAG} and \ref{OHM} merges. This would directly prove the point convergence of OHM.

\begin{theorem}
    Let $\Z \subseteq \R^n$ be a closed convex set,  $F: \Z \rightarrow \Z$ be $L$-Lipschitz, and $F+ N_\Z$ be $\rho$-comonotone with $\rho \ge -\frac{1}{20L}$. Let $\beta_k= \frac{1}{k+1}$,  $\eta = \frac{0.31}{L}$, and $w_0 = z_0$. Then the iterates $\{z_k\}_{k\geq 0}$ of \ref{proj-EAG} and the iterates $\{w_k\}_{k\geq 0}$ of \ref{OHM} merges, i.e., 
    \[
    \InNorms{z_{k+1} - w_{k+1}}^2 \le \frac{40000 H^2}{(k+1)^{\half}} + \frac{30000H^2}{(k+1)^2} = O(\frac{H^2}{(k+1)^{\half}})
    \]
    In particular, $\lim_{k \rightarrow \infty} \InNorms{z_{k+1} - w_{k+1}} = 0$.
\end{theorem}
\begin{proof}
    Using the update rule of \ref{proj-EAG} and \ref{OHM} and recall: $\eta c_{k+1} := z_{k} - \eta F(z_{k+\half}) + \beta_k (z_0 - z_{k}) - z_{k+1}$ and $\eta d_{k+1} = w_{k+\half} - \eta F(w_{k+1}) - w_{k+1}$ we have
    \begin{align*}
        z_{k+1} - w_{k+1} &= \InParentheses{\beta_k z_0 + (1-\beta_k)z_k - \eta F(z_{k+\half}) - \eta c_{k+1}} - \InParentheses{\beta_k w_0 + (1-\beta_k)w_k -\eta F(w_{k+1}) - \eta d_{k+1}}\\
        & = (1-\beta_k) (z_k - w_k) + \eta \InParentheses{F(w_{k+1}) + d_{k+1} - F(z_{k+\half}) - c_{k+1}}, \tag{$w_0 = z_0$}
    \end{align*}
    which implies
    \begin{align*}
        \InNorms{z_{k+1} - w_{k+1}}^2 &= (1-\beta_k)^2 \InNorms{z_k - w_k}^2 +  \eta^2 \InNorms{F(w_{k+1}) + d_{k+1} - F(z_{k+\half}) - c_{k+1}}^2 \\
        & \quad + \underbrace{{2\InAngles{(1-\beta_k)(z_k - w_k), \eta \InParentheses{F(w_{k+1}) + d_{k+1} - F(z_{k+\half}) - c_{k+1}}}}}_{\textbf{I}}.
    \end{align*}
    We focus on term \textbf{I}. Denote $c_{k+\half}:= \frac{\beta_k z_0 + (1-\beta_k)z_k - \eta F(z_k) - z_{k+\half}}{\eta} \in N_\Z(z_{k+\half})$. We can verify
    \[
    z_{k+\half} - w_{k+1} = (1-\beta_k)(z_k - w_k) - \eta\InParentheses{ F(z_k) + c_{k+\half} -  F(w_{k+1}) -  d_{k+1}}. 
    \]
    Thus term \textbf{I} can be rewritten as 
    \begin{align*}
        \textbf{I} &= 2\InAngles{(1-\beta_k)(z_k - w_k), \eta \InParentheses{F(w_{k+1}) + d_{k+1} - F(z_{k+\half}) - c_{k+1}}} \\
        &=2\InAngles{ z_{k+\half} - w_{k+1} + \eta\InParentheses{ F(z_k) + c_{k+\half} -  F(w_{k+1}) -  d_{k+1}}, \eta \InParentheses{F(w_{k+1}) + d_{k+1} - F(z_{k+\half}) - c_{k+1}}} \\
        & = 2\eta^2 \InAngles{F(z_k) + c_{k+\half} -  F(w_{k+1}) -  d_{k+1}, F(w_{k+1}) + d_{k+1} - F(z_{k+\half}) - c_{k+1}} + \\
        &\quad + \underbrace{2 \eta \InAngles{z_{k+\half} - w_{k+1},  F(w_{k+1}) + d_{k+1} - F(z_{k+\half}) - c_{k+1} } }_{\textbf{$A_k$}}.
    \end{align*}
    Combining the above (we keep term \textbf{$A_k$} for now), we get 
    \begin{align*}
        &\InNorms{z_{k+1} - w_{k+1}}^2\\
        &\le (1-\beta_k)^2 \InNorms{z_k - w_k}^2 +  \eta^2 \InNorms{F(w_{k+1}) + d_{k+1} - F(z_{k+\half}) - c_{k+1}}^2 \\
        &\quad + 2\eta^2 \InAngles{F(z_k) + c_{k+\half} -  F(w_{k+1}) -  d_{k+1}, F(w_{k+1}) + d_{k+1} - F(z_{k+\half}) - c_{k+1}} + \textbf{$A_k$} \\
        &= (1-\beta_k)^2 \InNorms{z_k - w_k}^2 - \eta^2 \InNorms{F(w_{k+1}) + d_{k+1}}^2 + 2\eta^2 \InAngles{F(z_k) + c_{k+\half}, F(w_{k+1}) + d_{k+1}  } \\
        &\quad - 2\eta^2 \InAngles{F(z_k) + c_{k+\half}, F(z_{k+\half}) + c_{k+1}} + \eta^2 \InNorms{F(z_{k+\half}) + c_{k+1}}^2 + \textbf{$A_k$}. \\
        &\le  (1-\beta_k)^2 \InNorms{z_k - w_k}^2 + \eta^2 \InNorms{F(z_k) + c_{k+\half}}^2 \tag{We use $-a^2 + 2ab \le b^2$} \\
        &\quad - 2\eta^2 \InAngles{F(z_k) + c_{k+\half}, F(z_{k+\half}) + c_{k+1}} + \eta^2 \InNorms{F(z_{k+\half}) + c_{k+1}}^2 + \textbf{$A_k$} \\
        & = (1-\beta_k)^2 \InNorms{z_k - w_k}^2 + \eta^2 \InNorms{F(z_k) + c_{k+\half} - F(z_{k+\half}) - c_{k+1}}^2 + \textbf{$A_k$} \\
        & = (1-\beta_k)^2 \InNorms{z_k - w_k}^2 +  \InNorms{z_{k+1} - z_{k+\half}}^2 + \textbf{$A_k$}
    \end{align*}
    Plugging $\beta_k = \frac{1}{k+1}$ and multiplying both sides with $(k+1)^2$ gives
    \[
    (k+1)^2\InNorms{z_{k+1} - w_{k+1}}^2 \le k^2\InNorms{z_k - w_k}^2 + (k+1)^2 \InNorms{z_{k+1} - z_{k+\half}}^2 + (k+1)^2\textbf{$A_k$}.
    \]
    Telescoping the above gives 
    \begin{align*}
        &(k+1)^2\InNorms{z_{k+1} - w_{k+1}}^2 \le \InNorms{z_1 - w_1}^2 + \sum_{t=1}^k (t+1)^2 \InNorms{z_{t+1} - z_{t+\half}}^2 + \sum_{t=1}^k (t+1)^2 \textbf{$A_t$}. \\ 
        &\Rightarrow \InNorms{z_{k+1} - w_{k+1}}^2 \le \frac{\InNorms{z_1 - w_1}^2}{(k+1)^2} + \frac{1}{(k+1)^2}\sum_{t=1}^k (t+1)^2 \InNorms{z_{t+1} - z_{t+\half}}^2 + \frac{1}{(k+1)^2}\sum_{t=1}^k (t+1)^2 \textbf{$A_t$}
    \end{align*}
    It remains to show that $\frac{1}{(k+1)^2}\sum_{t=1}^k (t+1)^2 \InNorms{z_{t+1} - z_{t+\half}}^2 = o(1)$ and $\frac{1}{(k+1)^2}\sum_{t=1}^k (t+1)^2 \textbf{$A_t$} = o(1)$. 
    
    For the first term, using \Cref{corollary:proj-EAG bounded summation}, we have
    \begin{align*}
        \sum_{t=1}^k (t+1)^2 \InNorms{z_{t+1} - z_{t+\half}}^2 &\le 2 \sum_{t=1}^k t(t+1) \InNorms{z_{t+1} - z_{t+\half}}^2 \\
        &\le 8000 H^2.
    \end{align*}

    We need a more careful analysis of $A_k$ for the second term. We decompose $A_k = B_k + C_k + D_k$ as follows.
    \begin{align*}
        A_k &= 2 \eta \InAngles{z_{k+\half} - w_{k+1},  F(w_{k+1}) + d_{k+1} - F(z_{k+\half}) - c_{k+1}} \\
        &= 2 \eta \InAngles{z_{k+\half} -  z_{k+1},  F(w_{k+1}) + d_{k+1} - F(z_{k+\half}) - c_{k+1}} \\
        &\quad + 2\eta \InAngles{z_{k+1} - w_{k+1}, F(w_{k+1}) + d_{k+1} - F(z_{k+\half}) - c_{k+1}} \\
        &= \underbrace{2 \eta \InAngles{z_{k+\half} - z_{k+1},  F(w_{k+1}) + d_{k+1} - F(z_{k+\half}) - c_{k+1}}}_{B_k} + \underbrace{2\eta \InAngles{z_{k+1} - w_{k+1}, F(z_{k+1}) - F(z_{k+\half})}}_{C_k} \\
        &\quad  + \underbrace{2\eta \InAngles{ z_{k+1} - w_{k+1}, F(w_{k+1}) + d_{k+1} - F(z_{k+1}) - c_{k+1} }}_{D_k}.
    \end{align*}

     For $B_k$, we have
    \begin{align*}
        B_k &= 2\eta \InAngles{ z_{k+\half} - z_{k+1}, F(w_{k+1}) + d_{k+1} - F(z_{k+\half}) - c_{k+1} } \\
        &\le 2 \eta \InNorms{z_{k+1} - z_{k+\half}} \InParentheses{\InNorms{F(w_{k+1}) + d_{k+1}} + \InNorms{F(z_{k+1}) - F(z_{k+\half})}+\InNorms{F(z_{k+1}) + c_{k+1}}}\\
        & \le 2\eta L \InNorms{z_{k+1} - z_{k+\half}}^2 + 2 \eta \InNorms{z_{k+1} - z_{k+\half}} \InParentheses{\InNorms{F(w_{k+1}) + d_{k+1}} +\InNorms{F(z_{k+1}) + c_{k+1}}}\\
        &\le  \frac{16000 H^2}{(k+1)^2} + 2\eta \sqrt{\frac{8000 H^2}{(k+1)^2}} \cdot \InParentheses{ \frac{2\InNorms{z_0 - z^*}}{\eta (k+1)} + \frac{\sqrt{20 H^2}}{\eta(k+1)} } \\
        &\le \frac{17140 H^2}{(k+1)^2}.
    \end{align*}
    where in the second last inequality we use the last-iterate convergence rate of \ref{OHM} and \ref{proj-EAG} (\Cref{thm:proj-EAG last rate}) as well as the bound on $\InNorms{z_{k+1} - z_{k+\half}}^2$ (\Cref{corollary:proj-EAG bounded summation}).

    For $C_k$, we have 
    \begin{align*}
        C_k &= 2\eta \InAngles{z_{k+1} - w_{k+1}, F(z_{k+1}) - F(z_{k+\half})} \\
        &\le 2\eta \InNorms{z_{k+1} - w_{k+1}} \InNorms{ F(z_{k+1}) -  F(z_{k+\half})} \\
        &\le 2\InNorms{z_{k+1}- w_{k+1}}\InNorms{z_{k+1} - z_{k+\half}}\tag{$F$ is $L$-Lipschitz and $\eta L\le 1$}\\
        & \le 400H \InNorms{z_{k+1} - z_{k+\half}},
    \end{align*}
    where in the last inequality we use $\InNorms{z_{k+1}- w_{k+1}} \le 90H$ is bounded by \Cref{coro:proj-EAG OHM bounded}.

    For $D_k$, since $F+ N_\Z$ is $\rho$-comonotone, we have
    \begin{align*}
        D_k &= 2 \eta \InAngles{z_{k+1} - w_{k+1},  F(w_{k+1}) + d_{k+1} - F(z_{k+1}) - c_{k+1}} \\
        &\le -2\rho\eta \InNorms{F(w_{k+1}) + d_{k+1} - F(z_{k+1}) - c_{k+1}}^2 \\
        &\le -4\rho\eta \InNorms{F(w_{k+1}) + d_{k+1}}^2-4\rho\eta \InNorms{F(z_{k+1}) + c_{k+1}}^2\\
        &\le  -\frac{4\rho}{\eta} \cdot \frac{4\InNorms{z_0 - z^*}^2}{(k+1)^2} -\frac{4\rho}{\eta} \cdot \frac{20 H^2}{(k+1)^2} \tag{convergence rate of \ref{OHM} and \Cref{thm:proj-EAG last rate}} \\
        &\le \frac{24 H^2}{(k+1)^2}. \tag{$\frac{\rho}{\eta} \ge - \frac{1}{4}$}
    \end{align*}
    
    Combining the above bounds for $B_k, C_k, D_k$, we get
    \begin{align*}
        (k+1)^2 A_k \le 17424 H^2 + 400H(k+1)^2\InNorms{z_{k+1} - z_{k+\half}}.
    \end{align*}
    Since $\sum_{t=1}^\infty (t+1)^2 \InNorms{z_{t+1} - z_{t+\half}}^2 \le 3200H^2$, using Cauchy-Schwartz, we get
    \begin{align*}
         \sum_{t=1}^k (t+1)^2 \InNorms{z_{t+1} - z_{t+\half}} &\le \sqrt{ \InParentheses{\sum_{t=1}^k (t+1)^2 \InNorms{z_{t+1} - z_{t+\half}}^2} \cdot \InParentheses{\sum_{t=1}^k (t+1)^2}} \\
         &\le 100H (k+1)^{3/2}.
    \end{align*} 
    Thus $\frac{1}{(k+1)^2} \sum_{t=1}^k (t+1)^2 \InNorms{z_{t+1} - z_{t+\half}} \le 100H (k+1)^{-\frac{1}{2}}$.
    
    Hence we get 
    \begin{align*}
        \frac{1}{(k+1)^2} \sum_{t=1}^k (t+1)^2 A_t = \frac{17424 H^2}{(k+1)^2} + \frac{40000 H^2}{(k+1)^{\half}}
    \end{align*}
    Combining all the above, we conclude that for all $k$, 
    \begin{align*}
        \InNorms{z_{k+1} - w_{k+1}}^2 &\le \frac{\InNorms{z_1 - w_1}^2}{(k+1)^2} + \frac{1}{(k+1)^2}\sum_{t=1}^k (t+1)^2 \InNorms{z_{t+1} - z_{t+\half}}^2 + \frac{1}{(k+1)^2}\sum_{t=1}^k (t+1)^2 \textbf{$A_t$} \\
        &\le \frac{40000 H^2}{(k+1)^{\half}} + \frac{30000H^2}{(k+1)^2}.
    \end{align*}
    This completes the proof. 
\end{proof}

\begin{lemma}[Bounded Iterates of \ref{proj-EAG}]
\label{lemma:proj-EAG bounded iterates}
    The iterates $\{z_k\}_{k \ge 0}$ of \ref{proj-EAG} satisfies for all $k \ge 0$
    \begin{align*}
        \InNorms{z_{k}- z_0}^2 &\le 16040 H^2.
    \end{align*}
\end{lemma}
\begin{proof}
    For $k = 1$, the proof of \Cref{lemma:proj-EAG V_1 bound} gives $\InNorms{z_1 - z_0}^2 \le H^2$.  By definitions of \ref{proj-EAG} and $c_{k+1}$, we have
    \begin{align*}
        \InNorms{z_{k+1} - z_0}^2 &= \InNorms{\frac{k}{k+1}(z_k - z_0) - \eta( F(z_{k+\half}) + d_{k+1})}^2 \\
        &\le \InParentheses{1 + \frac{1}{k}}\frac{k^2}{(k+1)^2} \InNorms{z_k - z_0}^2 +  (1 + k) \eta^2 \InNorms{( F(z_{k+\half}) + d_{k+1})}^2 \tag{Young's inequality} \\
        &\le \frac{k}{k+1} \InNorms{z_k - z_0}^2 + 2(1+k)\eta^2 \InParentheses{ \InNorms{F(z_{k+\half}) - F(z_{k+1})}^2 + \InNorms{F(z_{k+1}) + d_{k+1}}^2 } \\
        &\le \frac{k}{k+1} \InNorms{z_k - z_0}^2 + \frac{1}{k+1}16040 H^2.\tag{\Cref{corollary:proj-EAG bounded summation} and \Cref{thm:proj-EAG last rate}}
    \end{align*}
    By induction, we get $\InNorms{z_{k+1} - z_0}^2 \le 16040 H^2$.
\end{proof}
\begin{corollary}\label{coro:proj-EAG OHM bounded}
    Let $w_0 = z_0$. The iterates $\{z_k\}_{k \ge 0}$ of \ref{proj-EAG} and the iterates $\{w_k\}_{k \ge 0}$ of \ref{OHM} satisfies for all $k \ge 0$,
    \begin{align*}
        \InNorms{w_{k} - z_{k}} \le \sqrt{(32080 + 8) H^2} \le 200H.
    \end{align*}
\end{corollary}

\section{Auxiliary Propositions}\label{sec:auxiliary}

\begin{remark}
    The proofs for our algorithms' last-iterate convergence rates are all based on potential function arguments. The core of these arguments is to prove the potential function's (approximate) monotonicity. The three identities below simplify tedious algebraic calculations and are useful to show the potential functions' (approximate) monotonicity. 
\end{remark}
\begin{proposition}\label{prop:EAG identity}
Let $x_0, x_1, x_2, x_3, y_1, y_2, y_3, u_1, u_3$ be arbitrary vectors in $\R^n$. Let $q > 0$, $p > 0$, and $r \ne -\frac{1}{2}$ be real numbers. 

If $u_3 = x_1 - y_2 + \frac{1}{q+1}(x_0 - x_1) - x_3$, then the following identity holds: 
\begin{nalign}\label{eq:identity proj-EAG}
    &\frac{q(q+1)}{2} \cdot \InNorms{y_1 + u_1}^2 + q\cdot \InAngles{y_1+u_1, x_1 - x_0} \\
    &- \InParentheses{ \frac{(q+1)(q+2)}{2} \cdot \InNorms{y_3 + u_3}^2 + (q+1)\cdot \InAngles{y_3+u_3, x_3 - x_0}}\\
    &-\frac{q(q+1)}{2p} \cdot \InParentheses{p \cdot \InNorms{x_2 - x_3}^2 - \InNorms{y_2 - y_3}^2} \\
    & - q(q+1) \cdot \InAngles{y_3 - y_1, x_3 - x_1} \\
    &-q(q+1) \cdot \InAngles{x_1 - y_1 - x_2 + \frac{1}{q+1}(x_0 - x_1), x_2 - x_3} \\
    & -q(q+1) \cdot \InAngles{u_3, x_3 - x_1} \\
    & -q(q+1) \cdot \InAngles{u_1, x_1 - x_2} \\
    =& \frac{q(q+1)}{2} \cdot \InNorms{x_2 -x_1 + y_1 + u_1 + \frac{1}{q+1}(x_1 - x_0)}^2 \\
    & + \frac{(1-p)q(q+1)}{2p} \cdot \InNorms{y_2 - y_3}^2 \\
    & + (q+1) \cdot \InAngles{y_2 - y_3, y_3 + u_3 }
\end{nalign}

If $u_1 = x_1 - y_1 + \frac{1}{q+1}(x_0 - x_1) - x_2$ and $u_3 = x_1 - y_2 + \frac{1}{q+1}(x_0 - x_1) - x_3$, then the following identity holds:
\begin{nalign}\label{eq:identity composite-EAG}
    &\frac{q(q+1)}{2} \cdot \InNorms{y_1 + u_1}^2 + q\cdot \InAngles{y_1+u_1, x_1 - x_0} \\
    &- \InParentheses{ \frac{(q+1)(q+2)}{2} \cdot \InNorms{y_3 + u_3}^2 + (q+1)\cdot \InAngles{y_3+u_3, x_3 - x_0}}\\
    &-\frac{q(q+1)}{2p} \cdot \InParentheses{p \cdot \InNorms{x_2 - x_3}^2 - \InNorms{y_2 - y_3}^2} \\
    & - q(q+1) \cdot \InAngles{y_3 + u_3
    - y_1 - u_1, x_3 - x_1} \\
    =& \frac{(1-p)q(q+1)}{2p} \cdot \InNorms{y_2 - y_3}^2 \\
    & + (q+1) \cdot \InAngles{y_2 - y_3, y_3 + u_3 }
\end{nalign}

If $u_1 = \frac{x_1 + \frac{1}{q+1}(x_0 - x_1) - \frac{q}{q+1}(1+2r)y_1 - x_2}{\frac{q}{q+1}(1+2r)}$ and $u_3 = x_1 +  \frac{1}{q+1}(x_0 - x_1) - y_2 - \frac{2rq}{q+1}(y_1 + u_1) - x_3$, then the following identity holds:
\begin{nalign}\label{identity AS}
    &\InParentheses{\frac{q^2}{2}(1+2r)-rq}\cdot \InNorms{y_1 + u_1}^2 + q \cdot \InAngles{y_1 + u_1, x_1 - x_0}  \\
    & - \InParentheses{\frac{(q+1)^2}{2}(1+2r)-r(q+1)}\cdot \InNorms{y_3 + u_3}^2 - (q+1) \cdot \InAngles{y_3 + u_3, x_3 - x_0} \\
    & - \frac{(q+1)^2}{2} \cdot \InParentheses{ p \cdot \InNorms{x_2 - x_3}^2 - \InNorms{y_2 - y_3}^2} \\
    & - q(q+1) \cdot \InParentheses{\InAngles{y_3 + u_3 - y_1 - u_1, x_3 - x_1} - r \InNorms{y_3 + u_3 - y_1 - u_1}^2} \\
    = & \frac{(1-p)(q+1)^2}{2} \cdot \InNorms{x_2 - x_3}^2.
\end{nalign}

\end{proposition}
\begin{proof}
 We verify the three identities using {\sc Matlab}. Readers can find the verification code at \url{https://github.com/weiqiangzheng1999/Accelerated-Comonotone-Inclusion}. 
\end{proof}

\begin{proposition}\label{prop:sequence analysis}
Let $\{a_k \in \R^{+}\}_{k\ge 2}$ be a sequence of real numbers. Let $C_1 \ge 0$ and $p \in (0, \frac{1}{3})$ be two real numbers. If the following condition holds for every $k\geq 2$,
\begin{align}\label{eq:sequence condition}
    \frac{k^2}{4} \cdot a_k \le C_1 + \frac{p}{1-p} \cdot \sum_{t=2}^{k-1} a_t,
\end{align}
then for each $k\geq 2$ we have 
\begin{align}\label{eq:induction assumption}
    a_k \le \frac{4\cdot C_1}{1-3p} \cdot \frac{1}{k^2}.
\end{align}

\end{proposition}
\begin{proof}
We prove the statement by induction.

\noindent{\bf Base Case: $k = 2$. } From Inequality~\eqref{eq:sequence condition}, we have 
\begin{align*}
    \frac{2^2}{4} \cdot a_2 \le C_1\quad \Rightarrow \quad a_2 \le C_1 \le \frac{4\cdot C_1}{1-3p} \cdot \frac{1}{2^2}.
\end{align*}
Thus, Inequality~\eqref{eq:induction assumption} holds for $k= 2$. 

\noindent{\bf Inductive Step: for any $k \ge 3$. } 
Fix some $k \ge 3$ and assume that Inequality~\eqref{eq:induction assumption} holds for all $2 \le t \le k -1$. We slightly abuse notation and treat the summation in the form $\sum_{t=3}^{2}$ as 0.
By Inequality~\eqref{eq:sequence condition}, we have
\begin{align*}
\frac{k^2}{4}\cdot a_k &\le C_1 + \frac{p}{1-p} \cdot \sum_{t = 2}^{k-1} a_t \\ 
&\le \frac{C_1}{1-p} +  \frac{p}{1-p} \cdot \sum_{t = 3}^{k-1} a_t \tag{$a_2 \le C_1$}\\
& \le \frac{C_1}{1-p} +  \frac{4p\cdot C_1 }{(1-p)(1-3p)} \cdot \sum_{t = 3}^{k-1} \frac{1}{t^2} \tag{Induction assumption on Inequality~\eqref{eq:induction assumption}}\\
& \le \frac{C_1}{1-p} +  \frac{2p\cdot C_1 }{(1-p)(1-3p)} \tag{$\sum_{t=3}^\infty \frac{1}{t^2} = \frac{\pi^2}{6}-\frac{5}{4} \le \frac{1}{2}$} \\
& = \frac{C_1}{1-3p}.
\end{align*}
This complete the inductive step. Therefore, for all $k\ge 2$, we have $a_k \le \frac{4\cdot C_1}{1-3p} \cdot \frac{1}{k^2}$.
\end{proof}

\begin{proposition}\label{prop:1/2 sequence analysis}
Let $\{a_k \in \R^{+}\}_{k\ge 2}$ be a sequence of real numbers and $C_1 \ge 0$. If the following condition holds for every $k\geq 2$,
\begin{align}\label{eq:1/2 sequence condition}
    \frac{k^2}{4} \cdot a_k \le C_1 + \frac{1}{2} \cdot \sum_{t=2}^{k-1} a_t,
\end{align}
then for each $k\geq 2$ we have 
\begin{align}\label{eq:1/2 induction assumption}
    a_k \le \frac{20C_1}{k^2}
\end{align}

\end{proposition}
\begin{proof}
We prove the statement by induction.

\noindent{\bf Base Case: $k = 2, 3, 4$. } From Inequality~\eqref{eq:1/2 sequence condition}, we can directly calculate that $a_2, a_3, a_4 \le C_1$.
Thus, Inequality~\eqref{eq:1/2 induction assumption} holds for $k= 2, 3,4$. 

\noindent{\bf Inductive Step: for any $k \ge 5$. } 
Fix some $k \ge 5$ and assume that Inequality~\eqref{eq:1/2 induction assumption} holds for all $2 \le t \le k -1$. We slightly abuse notation and treat the summation in the form $\sum_{t=5}^{4}$ as 0.
By Inequality~\eqref{eq:1/2 sequence condition}, we have
\begin{align*}
\frac{k^2}{4}\cdot a_k &\le C_1 +\frac{1}{2}\cdot \sum_{t = 2}^{k-1} a_t \\ 
&\le \frac{5}{2}C_1 +  \frac{1}{2} \cdot \sum_{t = 5}^{k-1} a_t \tag{$a_2, a_3, a_4 \le C_1$}\\
& \le \frac{5}{2}C_1 +  10C_1 \cdot \sum_{t = 5}^{k-1} \frac{1}{t^2} \tag{Induction assumption on Inequality~\eqref{eq:induction assumption}}\\
& \le 5C_1. \tag{$\sum_{t=5}^\infty \frac{1}{t^2} =  \le \frac{1}{4}$} \\
\end{align*}
This complete the inductive step. Therefore, for all $k\ge 2$, we have $a_k \le \frac{20C_1}{k^2}$.
\end{proof}

\end{document}